\documentclass[11pt]{amsart}
\usepackage{amssymb,amsmath}
\usepackage{amscd}
\usepackage{tikz}
\usetikzlibrary{arrows,shapes}
\usetikzlibrary{decorations.markings}
\usepackage{amssymb}
\usepackage{latexsym}
\input amssym.def
\usepackage[applemac]{inputenc}

\font\tenmsb=msbm10 \font\sevenmsb=msbm7 \font\fivemsb=msbm5
\newfam\msbfam
\textfont\msbfam=\tenmsb \scriptfont\msbfam=\sevenmsb
\scriptscriptfont\msbfam=\fivemsb

\font\teneufm=eufm10 \font\seveneufm=eufm7 \font\fiveeufm=eufm5
\newfam\eufmfam
\textfont\eufmfam=\teneufm \scriptfont\eufmfam=\seveneufm
\scriptscriptfont\eufmfam=\fiveeufm

\def\co#1{
}

\renewcommand{\epsilon}{\varepsilon}
\renewcommand{\setminus}{\smallsetminus}

\newcommand{\Z}{\mathbb Z}
\newcommand{\N}{\mathbb N}

\newcommand{\SL}{\operatorname{SL}}

\newcommand{\FP}{\operatorname{FP}}

\newcommand{\cohom}[3]{H^{{\raise1pt\hbox{$\scriptstyle#1$}}}(#2\>\!,#3)}
\newcommand{\tatecohom}[3]%
  {\widehat H^{{\raise1pt\hbox{$\scriptstyle#1$}}}(#2\>\!,#3)}

\newcommand{\Cohom}[3]%
  {H^{{\raise1pt\hbox{$\scriptstyle#1$}}}\big(#2\>\!,#3\big)}
\newcommand{\Tatecohom}[3]%
  {\widehat H^{{\raise1pt\hbox{$\scriptstyle#1$}}}\big(#2\>\!,#3\big)}

\newcommand{\homol}[3]{H_{{\lower1pt\hbox{$\scriptstyle#1$}}}(#2\>\!,#3)}
\newcommand{\homolog}[2]{H_{{\lower1pt\hbox{$\scriptstyle#1$}}}(#2)}

\newcommand{\im}{\operatorname{Im}}

\newcommand{\ra}{\rightarrow}

\newcommand{\nfset}{{\mathcal{N}}}
\newcommand{\pres}{{\mathcal{P}}}
\def\nf#1{\mathsf{nf}(#1)}    
\def\nfv#1{\mathsf{nf}(#1)}    
    
\newcommand{\tria}{$\nfset$-triangular}
\newcommand{\trig}{$\nfset_G$-triangular}
\newcommand{\trip}{$\nfset_{G_p}$-triangular}
\newcommand{\ft}{fully \tria}
\newcommand{\ftg}{fully \trig}
\newcommand{\ftp}{fully \trip}
\newcommand{\nformed}{$\nfset$-labeled}  
\renewcommand{\ra}{\rightarrow}

\newcommand{\bo}{{\partial}}

\newcommand{\ff}{\Phi}  
\newcommand{\ga}{\Gamma}  
\newcommand{\pp}{\mathcal{P}}
\newcommand{\dd}{\Delta}
\newcommand{\lbl}{{\mathsf{label}}}  
\newcommand{\rep}{\rho}
\newcommand{\path}{{\mathsf{path}}}  
\newcommand{\tree}{{\mathcal{T}}}  
\renewcommand{\ff}{\Phi} 
\newcommand{\sff}{\phi} 
\newcommand{\emptyword}{1}  
\newcommand{\groupid}{\epsilon}  
\newcommand{\graph}{{\mathsf{graph}}}
\newcommand{\proj}{{\mathsf{proj}}}
\newcommand{\atob}{\psi}  
\newcommand{\tail}{{\mathsf{tl}}}   
\newcommand{\Tail}{{\mathsf{Tail}}} 
\newcommand{\head}{{\mathsf{hd}}}   
\newcommand{\trans}{{\mathsf{trans}}}  
\newcommand{\subg}{{\mathsf{sub}}}  
\newcommand{\last}{{\mathsf{last}}}  
\newcommand{\cprs}{complete prefix-rewriting system}
\newcommand{\prs}{prefix-rewriting system}
\newcommand{\syreg}{synchronously regular}
\newcommand{\pad}{{\mathsf{pad}}}
\newcommand{\ism}{\hat\rho}  
\newcommand{\bj}{\rho}  
\newcommand{\occ}{{\mathsf{occ}}}

\newcommand{\ppi}{\atob}

\newtheorem{theorem}{Theorem}[section]
\newtheorem{corollary}[theorem]{Corollary}
\newtheorem{proposition}[theorem]{Proposition}
\newtheorem{lemma}[theorem]{Lemma}
\newtheorem{definition}[theorem]{Definition}

\newtheorem{notation}[theorem]{Notation}

\newtheorem{remark}[theorem]{Remark}

\begin{document}

\title{HNN extensions and stackable groups}

\author[S.~Hermiller]{Susan Hermiller}
\address{Department of Mathematics\\
        University of Nebraska\\
         Lincoln NE 68588-0130, USA}
\email{hermiller@unl.edu}

\author[C.~Mart\'inez-P\'erez]{Conchita Mart\'inez-P\'erez}
\address{Departamento de Matem\'ticas\\
        Universidad de Zaragoza\\
         50009 Zaragoza, Spain}
\email{conmar@unizar.es}

\thanks{2010 {\em Mathematics Subject Classification}. 20F65; 20F10, 20F16, 68Q42}

\begin{abstract}
Stackability for finitely presented groups consists
of a dynamical system that iteratively moves paths
into a maximal tree in the Cayley graph.  Combining
with formal language theoretic restrictions  
yields auto- or algorithmic stackability,
which implies solvability of the word problem. 
In this paper we give two new characterizations of
the stackable property for groups, and
use these to show that every HNN extension
of a stackable group is stackable.  We apply
this to exhibit a wide range of Dehn functions
that are admitted by stackable and autostackable
groups, as well as an example of a 
stackable group with unsolvable word problem. 
We use similar methods to show that there exist finitely presented
metabelian groups that are non-constructible
but admit an autostackable structure.
\end{abstract}

\maketitle


\section{Introduction}\label{sec:intro}


Autostackability of finitely generated groups is
a topological property of the Cayley graph combined
with formal language theoretic restrictions, which is
an extension of the notions of automatic
groups and groups with finite complete
rewriting systems, introduced by Brittenham, Hermiller
and Holt in~\cite{bhj}.
An autostackable structure for a finitely generated
group implies a finite presentation,
a solution to the word problem, and
a recursive algorithm for building van Kampen
diagrams~\cite{MarkSusan}.
Moreover, in contrast to automatic groups,
Brittenham and Hermiller
together with Susse have shown that
the class of autostackable groups includes
all fundamental groups of 3-manifolds~\cite{bht},
 with Holt they have shown autostackable examples
of solvable groups that are not
virtually nilpotent~\cite{bhh}, and with Johnson
they show that Stallings' non-$FP_3$ group~\cite{bhj}
is autostackable.
In analogy with the relationship between automatic
and combable groups,
removing the formal language theoretic restriction
gives the stackable property for finitely
generated groups, and stackability implies
tame combability~\cite{bhtame}.
In this paper we give two new characterizations
of the stackability property, and determine closure of
stackability under HNN extensions.  We
then apply these results to a variety of 
examples to exhibit stackable groups that
are not algorithmically stackable and to
explore the Dehn functions of stackable, algorithmically
stackable, and autostackable groups. In the last section we also show that nonconstructible metabelian groups can be autostackable.

To make this more precise,
let $G$ be a group with a finite
inverse-closed generating set $X$, and   
let $\ga=\ga(G,X)$ be the associated
Cayley graph.  Denote the set of directed edges
in $\ga$ by $\vec E$, and the set of directed
edge paths by $\vec P$.
For each $g \in G$ and $a \in X$,
let $e_{g,a}$ denote the directed edge 
with initial vertex $g$, terminal vertex $ga$,
and label $a$;
we view the two directed edges $e_{g,a}$ 
and $e_{ga,a^{-1}}$  to
have a single underlying undirected edge in $\ga$.

A {\em flow function}
associated to a maximal tree $\tree$ in $\ga$ is a
function $\ff:\vec E \ra \vec P$ 
satisfying the properties that: 
\begin{itemize}
\item[(F1)] For each edge $e \in \vec E$,
the path $\ff(e)$ has the same initial and terminal
vertices as $e$.
\item[(F2d)] If the undirected edge underlying $e$ 
lies in the tree $\tree$, then $\ff(e)=e$.
\item[(F2r)]  The transitive closure 
$<_\ff$ of the relation $<$ on
$\vec E$ defined by 
\begin{itemize}
\item[]
$e' < e$ whenever $e'$ lies on the path $\ff(e)$
and the undirected edges underlying both
$e$ and $e'$ do not lie in $\tree$,
\end{itemize}
is a well-founded strict
partial ordering.
\end{itemize}
The flow function is {\em bounded} if there is
a constant $k$ such that for all $e \in \vec E$,
the path $\ff(e)$ has length at most $k$.
The map $\ff$ fixes the edges lying in the tree $T$
and describes a ``flow'' of the
non-tree edges toward the tree (or toward the basepoint);
starting from a non-tree edge and
iterating this function finitely many times results
in a path in the tree.

For each element $g \in G$, let $\nf{g}$ denote the
label of the unique geodesic (i.e., without backtracking)
path in the maximal tree $\tree$ from the identity element
$\groupid$ of $G$ to $g$, and let
$\nfset=\nfset_\tree:=\{\nf{g} \mid g \in G\}$ denote the 
set of these (unique) normal forms.
We use functions that pass between
paths and words by defining $\lbl:\vec P \ra X^*$ to be
the function that maps each directed path
to the word labeling that path and defining
$\path:\nfset \times X^* \ra \vec P$ to be
$\path(\nf{g},w) :=$ the path in $\ga$ that 
starts at $g$ and is labeled by $w$.
Observe that $\path(\nfset \times X) = \vec E$.

\begin{definition}\cite{MarkSusan,bhh}\label{def:autostackable}
Let $G$ be a group with a finite inverse-closed generating
set $X$.
\begin{enumerate}
\item The group $G$ is {\em stackable} over $X$ if there is a bounded
flow function on a maximal tree in the associated
Cayley graph. 
The {\em stacking map} is
$$\sff:=\lbl \circ \ff \circ \path: \nfset_\tree \times X \ra X^*.$$ 
\item The group $G$ is {\em algorithmically stackable}
over $X$ if $G$ admits a bounded flow function $\ff$ for which 
the graph 
$$
\graph(\sff):=\{(\nf{g},a,\phi(\nf{g},a)) \mid g \in G, a \in X\}
$$
of the stacking map $\phi$ is
decidable.
\item The group $G$ is {\em autostackable}
over $X$ if $G$ has a bounded flow function $\ff$
for which the graph 
of the associated stacking map is  synchronously regular.
\end{enumerate}
\end{definition}

A stackable group $G$ over a finite generating
set $A$ is finitely presented, with finite
presentation 
$R_\ff = \langle X \mid 
\{\sff(y,a) = a \mid y \in \nfset_\tree, a \in X\}\rangle$
(called the {\em stacking presentation})
associated to the flow function $\ff$.  

Each of these three stackability properties 
can also be stated in terms of prefix-rewriting systems.
A stackable structure is equivalent to
a bounded complete prefix-rewriting system
for $G$ over $X$, for which the irreducible
words are exactly the elements of the set
$\nfset_\tree$. A group is algorithmically
stackable (respectively, autostackable)
if and only if
it admits
a decidable (respectively, synchronously regular) 
bounded complete 
prefix-rewriting system~\cite{bhh}. 
(See Section~\ref{sec:notation} for 
definitions of these terms.)

In Section~\ref{sec:notation}, we begin with
notation and definitions we will use throughout
the paper.

Section~\ref{sec:fullytriangular} contains several 
characterizations
of stackability using properties of their
van Kampen diagrams, which we apply in our
proofs in later sections of this paper.

Section~\ref{sec:hnn} contains the 
proof of the following closure property for
the class of stackable groups with respect
to HNN extensions.

\smallskip

\noindent{\bf Theorem~\ref{thm:HNN}.}
{\em Let $H$ be a stackable group, 
let $A,B \leq H$ be finitely generated, and let $\atob:A\to B$ be
an isomorphism.  
Then the HNN extension $G=H\ast_\atob$ is also stackable. 
}

\smallskip

Corollary~\ref{cor:autohnn}
addresses closure of auto- and algorithmic stackability
of HNN extensions with additional constraints.
In the case of algorithmically stackable groups,
HNN extension closure can also be stated in terms of
a decision problem.
For a group $H$ with a finite inverse-closed
generating set $Y$ and subgroup $A$,
the {\em subgroup membership problem}
is decidable if there is an algorithm that,
upon input of any word $w$ over $Y$,
determines whether or not $w$ represents
an element of $A$.

\smallskip

\noindent{\bf Corollary~\ref{cor:alghnn}.}
{\em Let $H$ be an algorithmically stackable group, 
let $A,B \leq H$ be finitely generated, and let $\atob:A\to B$ be
an isomorphism.  
Suppose further that the subgroup membership problem
is decidable for the subgroups $A$ and $B$ in $H$.
Then the HNN extension $G=H\ast_\atob$ is also 
algorithmically stackable. 
}

\smallskip

We give three applications of Theorem~\ref{thm:HNN}
in Section~\ref{sec:stkapplications}.  
In the first, we apply Mihailova's~\cite{mihailova}
construction of subgroups of direct products of
free groups with unsolvable subgroup membership
problem, to show that stackability
and autostackability are not the same property.

\smallskip

\noindent{\bf Theorem~\ref{thm:stknotauto}.}
{\em There exists a stackable group with unsolvable word problem,
and hence stackability does not imply algorithmic stackability.}

\smallskip

Consequently the class of stackable groups
includes groups whose Dehn function is not computable.
Our second application is a proof that the
hydra groups of Dison and Riley~\cite{disonriley}
are algorithmically stackable, and consequently  
algorithmically stackable groups
admit extremely large Dehn functions.

\smallskip

\noindent{\bf Theorem~\ref{thm:hydra}.}
{\em The class of algorithmically stackable groups includes 
groups with Dehn functions in each level of
the Grzegorczyk hierarchy of primitive 
recursive functions.}

\smallskip

For the third application in Section~\ref{sec:stkapplications}, we
we study 
the nonmetabelian Baumslag group~\cite{baumslagnon} 
(also known as the Baumslag-Gersten group),
which is an HNN extension of a
Baumslag-Solitar group (which is
autostackable~\cite{bhh}).
 
\smallskip

\noindent{\bf Theorem~\ref{thm:nonmet}.}
{\em Baumslag's nonmetabelian group 
$\langle a,s \mid (sas^{-1})a(sa^{-1}s^{-1}) = a^2\rangle$
is autostackable.}

\smallskip

Platonov~\cite{platonov} (and in the case of
the lower bound, Gersten~\cite{gersten})
has shown that the Dehn function of the
Baumslag nonmetabelian group
is the nonelementary function
$n \ra \mathsf{tower}_2(\log_2(n))$, where
$\mathsf{tower}_2(1) = 2$ and $\mathsf{tower}_2(k)=2^{\mathsf{tower}_2(k-1)}$.
Although it is an open question whether
Baumslag's nonmetabelian group, or any other 
group with nonelementary
Dehn function, can have a finite complete
rewriting system, these results show that
groups with a bounded \syreg\ complete
prefix-rewriting system 
admit such Dehn functions.


\smallskip

\noindent{\bf Corollary~\ref{cor:autodehn}.}
{\em The class of autostackable groups 
includes groups with nonelementary Dehn functions.}

\smallskip


Finally, in Section~\ref{sec:metabelian}, we
consider metabelian groups.  Groves and Smith~\cite{grovessmith}
showed that a metabelian group $G$ has a finite
complete rewriting system if and only if
$G$ has the  homological finiteness condition
$FP_\infty$, and also if and only if $G$ is constructible
(that is, $G$ can be obtained from finite groups
by iteratively taking finite extensions, 
amalgamations, and finite rank HNN extensions).
Since \syreg\ bounded prefix-rewriting
systems (i.e., autostackable structures) are an
extension of finite complete rewriting systems,
it is natural to ask whether autostackability is also equivalent to 
$FP_\infty$ and constructibility in metabelian
groups.  

We consider Baumslag's~\cite{Baumslag1} example of a
finitely presented metabelian group whose commutator subgroup 
has infinite rank.  This group is finitely presented,
and hence has homological type $\FP_2$, but
the rank of  the commutator subgroup implies that
the group is not constructible, and hence
is not of type $\FP_\infty$.
For $p<\infty$, the $p$-torsion analog of
Baumslag's metabelian group is the Diestel-Leader
group $\Gamma_3(p)$, whose Cayley graph with
respect to a certain finite generating set is
the  Diestel-Leader graph $DL_3(p)$; for details
and more information, see the paper of
Stein, Taback, and Wong~\cite{stw} and the references there.  
The Diestel-Leader groups are finitely presented
metabelian groups that are not of type $\FP_3$\cite{bnw},
and hence also are nonconstructible metabelian groups. 
Each of these groups can be realized as an
HNN extension, but since the base group
of this extension is not finitely generated,
Theorem~\ref{thm:HNN} of Section~\ref{sec:hnn} 
does not apply in this case, and new methods 
are developed in Section~\ref{sec:metabelian}
to show the following.


\smallskip

\noindent{\bf Theorem~\ref{thm:metabelian}.}  {\em
Baumslag's metabelian group 
$G_\infty=\langle a,s,t\mid a^s=a^ta, [a^t,a]=1,[s,t]=1\rangle$
is algorithmically stackable, 
and the Diestel-Leader torsion analogs
$G_p=\langle a,s,t\mid a^s=a^ta, [a^t,a]=1,[s,t]=1,a^p=1\rangle$
with $p \ge 2$ are 
autostackable.
}

\smallskip

This shows that there exist nonconstructible
metabelian groups that admit a \syreg\ bounded prefix-rewriting
systems, giving a neagtive answer to the question above.

\smallskip

\noindent{\bf Corollary~\ref{cor:metabelian}.}  {\em
The class of autostackable groups contains
nonconstructible metabelian groups.
}

\smallskip

The authors thank Tim Susse for helpful
discussions and useful suggestions during
the course of this work.


\section{Notation and background}\label{sec:notation}


Throughout this paper, let $G$ be a group
with a finite {\em symmetric} generating set $X$; that
is, such that the generating set $X$ is closed under inversion.
Throughout the paper we assume that no element of $X$
represents the identity element of $G$.

Let $X^*$ denote the set of all words over $X$,
and let $X^+$ denote the set of all words except the 
empty word $\emptyword$.
A set $\nfset$ of {\em normal forms} for $G$ over $X$ is a 
subset of 
$X^*$ such that the restriction of the
canonical surjection $\rep: X^* \ra G$
to $\nfset$ is a bijection.
As in Section~\ref{sec:intro}, the symbol $\nf{g}$ denotes
the normal form for $g \in G$.  By slight abuse
of notation, we use the symbol $\nf{w}$ to denote the
normal form for $\rep(w)$ whenever $w \in X^*$.

For a word $w \in X^*$, we write $w^{-1}$ for the 
formal inverse of $w$ in $X^*$,
and let $l(w)$ denote
the length of the word $w$.
For words $v,w \in X^*$, we write $v=w$ if $v$
and $w$ are the same word in $X^*$, and write $v=_G w$ if
$v$ and $w$ represent the same element of $G$.

Let $\groupid$ denote the identity of $G$.
For $g,h \in G$, we use $g^h$ to denote the
conjugate $hgh^{-1}$ of $g$.

A {\em symmetrized} presentation 
$\pp = \langle X \mid R \rangle$ for $G$
satisfies the properties that
the generating set $X$ is symmetric
and the set $R$ of defining relations is closed under
inversion and cyclic conjugation.
Let $C$ be the Cayley 2-complex corresponding to this presentation,
whose 1-skeleton $C^1=\ga$ is the Cayley graph of $G$ over $A$.
For $g \in G$ and $x \in X$, let $e_{g,x}$ denote
the edge of $\ga$ labeled by $x$ with initial vertex $g$.
We consider the two directed edges
$e_{g,x}$ and $e_{gx,x^{-1}}$ to have the same
underlying directed edge in $\ga$ between 
the vertices $g$ and $gx$.


\subsection{Diagrams}\label{subsec:vkd}


For an arbitrary word $w$ in $X^*$
that represents the
trivial element $\groupid$ of $G$, there is a {\em van Kampen
diagram} $\dd$ for $w$ with respect to $\pp$.  
That is, $\dd$ is a finite,
planar, contractible combinatorial 2-complex with 
edges directed and
labeled by elements of $X$, satisfying the
properties that the boundary of 
$\dd$ is an edge path labeled by the
word $w$ starting at a basepoint 
vertex $*$ and
reading counterclockwise, and every 2-cell in $\dd$
has boundary labeled by an element of $R$.
(Note that we do not assume that van Kampen diagrams
in this paper are reduced; that is, we allow adjacent
2-cells in $\dd$ to be labeled by the same relator with
opposite orientations.)

For any van Kampen diagram
$\dd$ with basepoint $*$, 
let $\pi_\dd:\dd \ra C$
denote a cellular map such that $\pi_\dd(*)=\groupid$ and
$\pi_\dd$ maps edges to edges preserving both
label and direction.
Given $w\in X^*$, we denote by 
$w\Delta$ the diagram obtained by gluing the terminus of a path labeled 
by $w$ to the basepoint $*$ of $\Delta.$  


\subsection{Rewriting systems and languages}\label{subsec:rs}


The {\em regular} languages over a finite set $X$ are the
subsets of $X^*$ obtained from the finite subsets
of $X^*$ using finitely many operations from among
union, intersection, complement, 
concatenation ($S \cdot T := \{vw \mid v \in S$ and $w \in T\}$),
and Kleene star ($S^0:=\{\emptyword\}$, $S^n := S^{n-1} \cdot S$ and
$S^* := \cup_{n=0}^\infty S^n$).  
Equivalently, a subset $L \subseteq X^*$ is
regular if there is a monoid homomorphism
$\gamma:X^* \ra M$ for some finite monoid $M$,
such that $L$ is the preimage $L=\gamma^{-1}(S)$
for a subset $S$ of $M$.  Also equivalently,
a language $L \subseteq X^*$ is regular if
$L$ is the language accepted by a finite state
automaton.

A subset $L \subseteq (X^*)^n$ is called
a {\em \syreg} language if the padded extension set
$\{\pad(w) \mid w \in L\}$ is a regular language
over the finite alphabet $(X \cup \{\$\})^n$
(with $\$ \notin X$)
where $\pad(a_{1,1} \cdots a_{1,m_1},...,
a_{n,1} \cdots a_{n,m_n}) :=
((a_{1,1},...,a_{n,1}),...,(a_{1,N},...,a_{n,N}))$ for $N={\rm max}\{m_i\}$
whenever $a_{i,j} \in X$ for all $1 \le i \le n$
and $1 \le j \le m_i$ and $a_{i,j}=\$$ otherwise.

The class of regular
languages is closed under both image and
preimage via monoid homomorphisms and under quotients, and
the class of \syreg\ languages is closed with respect to
finite unions and intersections, Cartesian products,
and projection onto a single coordinate.

A language $L \subseteq X^*$ is {\em decidable}, also known as
{\em recursive}, if there is a Turing machine
that, upon input of any word $w$ over $X$, determines
(in a finite amount of time) whether or not $w \in L$.
The class of decidable languages is also 
closed under 
union, intersection, complement, 
concatenation,
Kleene star, and
image via monoid homomorphisms
(that map nonempty words to nonempty words).

See~\cite{echlpt} and~\cite{hu} for more information
about regular, \syreg, and decidable languages.

A {\em \cprs} for a group $G$ consists of a set $X$
and a set of rules $R \subseteq X^* \times X^*$
(with each $(u,v) \in R$ written $u \ra v$)
such that  $G$ is presented
(as a monoid) by 
$G = Mon\langle X \mid u=v$ whenever $u \ra v \in R \rangle,$
and the rewritings 
$uy \ra vy$ for all $y \in X^*$ and $u \ra v$ in $R$ satisfy:
\begin{itemize}
\item[(1)] there is no infinite chain $w \ra x_1 \ra x_2 \ra \cdots$
of rewritings, and
\item[(2)] each $g \in G$ is 
represented by exactly one irreducible word
over $X$.
\end{itemize}
The \prs\ is {\em bounded} if $X$ is finite
and there is a constant
$k$ such that for each pair $(u,v)$ in $R$
there are words $s,t,w \in X^*$ such that
$u=ws$, $v=wt$, and $l(s)+l(t) \le k$.
The \prs\ is {\em \syreg} 
if the set $X$ is finite and the
set of rules $R$ is \syreg. 

A {\em finite complete rewriting system}
for a group $G$
is a finite set $R' \subseteq X^* \times X^*$
presenting $G$ as a monoid, such that 
the rewritings 
$xuy \ra xvy$ for all $x,y \in X^*$ and $u \ra v$ in $R'$ satisfy
(1) and (2) above.  
Any finite complete rewriting system $R'$
has an associated \syreg\ bounded \cprs\  given by
$R=\{xu \ra xv \mid x \in X^*, (u,v) \in R'\}$.


\section{The \ft\ and \nformed\ properties}\label{sec:fullytriangular}


In this section we develop several conditions that
are equivalent to stackability, which will be used
to simply proofs in later sections of this paper.

Let $G$ be  a group with a finite 
presentation $\pres = \langle X \mid R\rangle$, where $X$ is inverse-closed
and $R$ is closed under inversion and cyclic
conjugation.  Let $\nfset$ be
a prefix-closed set of normal forms for $G$ over $X$.  
For each element $g$ in $G$, let $\nf{g}$ denote
the element of $\nfset$ representing $g$.
Let $\tree_\nfset$ denote the tree in the Cayley graph
$\Gamma$ for $G$ over $A$ consisting of the underlying
undirected edges that lie along paths from
the identity vertex $\groupid$ labeled by elements of $\nfset$.

\begin{definition}\label{def:triangular}
{\rm 
We say that a van Kampen diagram $\Delta$ is {\em \tria} 
if $\partial\Delta$ is labeled by a word of the form $\nf{g}x\nf{gx}^{-1}$ 
with $g \in G$ and $x \in X$. 
We refer to the paths $p_{lower}$, $p_x$, $p_{upper}$
in $\bo\Delta$ labeled by $\nf{g}$, $x$, and $\nf{gx}$
as the {\em lower normal form}, {\em isolated edge}, and
{\em upper normal form} of $\Delta$, respectively.
}
\end{definition}

\begin{definition}\label{def:nformed}
{\rm 
A \tria\ van Kampen diagram $\Delta$ is {\em \nformed} 
if for every vertex $v$ in the 0-skeleton $\Delta^{(0)}$
of $\Delta$, there is a path in $\Delta$
from the basepoint to $v$ labeled by the normal
form $\nf{\pi_{\Delta}(v)}$. In this case $\Delta$ determines a set of normal 
forms starting at the basepoint, namely 
$$
\nfv{\Delta} := \{\nf{\pi_\Delta(v)} \mid v \in \Delta^{(0)}\}.
$$ 

}
\end{definition}

If $G$ is stackable over the normal form set $\nfset$,
with stacking relations contained in $R$, 
then there is a recursive procedure for building van Kampen
diagrams for $G$ over the presentation $\pres$;
see~\cite{MarkSusan} for details of this
{\em stacking procedure}.  
In the following recursive definition we describe
a similar property.

\begin{definition}\label{fullytriangular} 
{\rm Fully \tria\ diagrams are recursively defined
as follows.
A diagram $\Delta$ is called
\begin{itemize}
\item[i)] 
{\em degenerate} if 
$\Delta$ has no 2-dimensional cells. 

\item[ii)] 
{\em minimal} if 
$\Delta$ has a single 2-dimensional cell $\sigma$
and $p_x \subset \partial\sigma$.

\item[iii)] 
{\em \ft} if $\Delta$ is either degenerate, minimal,
or there is a 2-dimensional cell
$\sigma$ in $\Delta$ with $p_x \subset \partial\sigma$, which we call
the {\sl isolated cell}, satisfying the following property.
If $e_1,...,e_t$ are the successive edges of the path
in $\partial\sigma \setminus p_x$ from the initial vertex
to the terminus of the edge $p_x$, then
for each $i = 1,...,t$ there is a 
\ft\ van Kampen diagram
$\Delta_i\subseteq\Delta$ having $e_i$ as isolated edge and the 
same basepoint as $\Delta$, such that
\begin{itemize} 
\item for each $i$, $\Delta_i\cap\Delta_{i+1}$ is both the upper normal form 
of $\Delta_i$ and the lower normal form of $\Delta_{i+1}$,
and 
\item $\Delta$ is the disjoint union of the $\Delta_i$ and $\sigma$, 
with the $\Delta_i$ glued
along these successive normal forms, and $\sigma$ glued
to the $\Delta_i$ along the edges $e_i$.
\end{itemize}
\end{itemize}
In this case,  $\Delta$ determines a set of \ft\ van Kampen diagrams, 
namely, the $\Delta_i$ together with the set of \ft\  
van Kampen diagrams determined by them; we denote this set by 
$\text{ft}(\Delta)$.
}
\end{definition}

We note that every \ft\ van Kampen diagram is also \nformed.
Also note that if $w\in X^*$ and  
$\Delta$ is \nformed\ or \ft, and moreover if  
for every $\mu\in \nfv{\Delta}$ we have $w\mu\in\mathcal{N}$,
then $w\Delta$ is also \nformed\ or \ft, respectively.

In the case that the group $G$ is stackable with respect to $\nfset$,
let $\Delta_{g,x}$ denote the van Kampen diagram with
boundary $\nf{g}x\nf{gx}^{-1}$ obtained using the stacking procedure.
Given a directed edge $e_{g,x}$ of the Cayley graph $\Gamma$,
if $e_{g,x}$ lies in the tree $\tree_\nfset$, 
then $\Delta_{g,x}$ is degenerate, and so
is \ft.
To check that for any edge $e_{g,x}$
not in $\tree_\nfset$ the diagram $\Delta_{g,x}$ 
is \ft, we note that 
the isolated cell has boundary $\Phi(e_{g,x})$ and $p_x$,
and the recursive property in Definition \ref{fullytriangular}~iii) 
holds true by induction using the partial order $<_\Phi$ on the set 
of recursive edges given by the stacking system.
In fact, 
if $e_{h,y} <_\Phi e_{g,x}$, then
$e_{h,y}$ is an edge needed in the process to transform $\nf{g}x$ into 
its normal form in the stacking
reduction procedure, and $\Delta_{h,y} \in \text{ft}(\Delta_{g,x})$. 

\begin{definition}\label{def:stksys}
{\rm A {\em stackable system of \ft\ van Kampen diagrams} 
(respect to $\pres$) is a set 
$$
\mathcal{S}:=\{\Delta_{g,x}\mid g \in G, x \in X
\}
$$
of \ft\ diagrams such that for each $g \in G$ and $x \in X$
the boundary of the diagram $\Delta_{g,x}$
is labeled by $\nf{g}x\nf{gx}^{-1}$, the diagram
$\Delta_{g,x}$ is degenerate if and only if $e_{g,x}$ is in
the tree $\tree_\nfset$, and 
whenever $\Delta_{g,x}\in\mathcal{S}$ contains more than
one 2-cell, the 
associated subdiagrams $\Delta_i$ in 
Definition \ref{fullytriangular}~iii) also belong to $\mathcal{S}$.}
\end{definition}

\begin{proposition}\label{equivalence}
The following are equivalent 
for a finitely presented group $G=\langle X \mid R\rangle$ 
with a prefix-closed normal form set $\nfset$ over $X$:
\begin{itemize}
\item[i)] $G$ is stackable with respect to $\mathcal{N}$.

\item[ii)] There is a stackable system of \ft\  van Kampen diagrams.

\item[iii)] For every 
$g \in G$ and $x\in X$ there is 
a \ft\  van Kampen diagram 
$\Delta_{g,x}$ with boundary $\nf{g}x\nf{gx}^{-1}$.

\item[iv)] For every $w \in \nfset$ and $x \in X$
there is a \nformed\  van Kampen diagram $\Delta_{w,x}'$
 with boundary $wx\nf{wx}^{-1}$.
\end{itemize}
\end{proposition}

\begin{proof} The fact that i) implies ii) follows from the discussion above,
and the implications ii) $\Rightarrow$ iii) 
and iii) $\Rightarrow$ iv) are immediate.  We prove the 
reverse implications in the same order. 

Let $\Gamma$ denote the Cayley graph of $G$ with
respect to the generating set $X$, and let
$\tree_\nfset$ be the maximal tree in $\Gamma$ 
traversed by paths starting at the identity vertex
and labeled by words in $\nfset$.

First, assume ii) holds, and let $\mathcal{S}$ be the
stackable system.  
If $e_{g,x}$  is in
the tree $\tree_\nfset$
in $\Gamma$, let  $\ff(e_{g,x}) := e_{g,x}$.
On the other hand, if $e_{g,x}$ does not lie in
the tree $\tree_\nfset$
in $\Gamma$ and $\Delta_{g,x}$ is 
the associated van Kampen diagram in $\mathcal{S}$, let $\ff(e_{g,x})$
be the path in $\Gamma$ labeled by
the word $\phi(e_{g,x})$ in $X^*$ such that
$\phi(e_{g,x})x^{-1}$ labels the isolated cell of $\Delta_{g,x}$. Note that this 
implies that $\phi(e_{g,x})x^{-1}$ is 
a relator, and thus the possible lengths of the words $\phi(e_{g,x})$ 
are bounded. Now, we define a map $\nu$ from the set of directed
edges in $\Gamma \setminus \tree_\nfset$ to $\Z^+$ by 
$\nu(e_{g,x}):=\text{ area of }\Delta_{g,x}$.  Use the function
$\nu$ and the usual ordering on $\Z$ to order the set of recursive edges;
this gives a well-founded ordering on $\vec E$
such that $e' <_\ff e$ implies $\nu(e') < \nu(e)$.
Thus the group is stackable by Definition~\ref{def:autostackable},
and i) holds.

Next assume that iii) holds. For each directed edge $e_{g,x}$ 
of the Cayley graph, let $n(e_{g,x})$ be the smallest area 
of a \ft\  van Kampen diagram with boundary $\nf{g}x\nf{gx}^{-1}$.
We define a stackable system $\mathcal{S}$ of \ft\ van Kampen
diagrams by induction
on $n(e_{g,x})$. To start, let $\mathcal{S}$ be the empty set.
Note that $n(e_{g,x})=0$ if and only if 
$e_{g,x}$ is in $\tree_\nfset$; in this case add
a degenerate diagram $\Delta(e_{g,x})$ to $\mathcal{S}$.
Similarly, $n(e_{g,x})=1$ if and only if
$e_{g,x}$ is not in $\tree_\nfset$ and there is a minimal
van Kampen diagram for $\nf{g}x\nf{gx}^{-1}$;
place a choice of such a diagram in $\mathcal{S}$.
Suppose now that $n(e_{g,x})>1$, and we have
diagrams $\Delta_{g,x}$ in $\mathcal(S)$ corresponding to
 all directed edges with lower value for the function $n$.
Let $\Delta'$ be a \ft\  van Kampen diagram
for $e_{g,x}$ with $n(e_{g,x})$ 2-cells.
Using Definition~\ref{fullytriangular}~iii), 
$\Delta'$ is the disjoint union of a 2-cell $\sigma$
containing $p_x$ and \ft\ diagrams $\Delta_i$,
with certain gluings.  The isolated edge $e_i$ associated
with $\Delta_i$ must satisfy $n(e_i)<n(e_{g,x})$.
For each $i$ we replace the subdiagram $\Delta_i$ of $\Delta'$ by
the diagram in $\mathcal{S}$ with the same boundary label,
to obtain a diagram $\Delta_{g,x}$; add this diagram to $\mathcal{S}$.
Then $\mathcal{S}$ is a stackable system,
completing the proof of ii).

Finally, assume that iv) holds. 
Then for each $g \in G$ and $x \in X$, 
there is a \nformed\ van Kampen diagram
with boundary label $\nf{g}x\nf{gx}^{-1}$;
from among all such diagrams,
let $\Delta_{\nf{g},x}'$ be a
diagram with the least possible number of 2-cells and
let $\alpha(g,x)$ denote this number.
We show that for each $g \in G$ and $x \in X$
there is a \ft\ van Kampen diagram $\Delta_{g,x}$
with boundary label $\nf{g}x\nf{gx}^{-1}$
by induction on $\alpha(g,x)$.

Suppose that $\alpha(g,x)=0$.  Then the
diagram $\Delta_{g,x}'$ is degenerate,
and hence \ft.  In this case we can
take $\Delta_{g,x}:=\Delta_{g,x}'$.

Suppose next that $n:=\alpha(g,x)>0$, and that
for all $g' \in G$ and $x' \in X$
with $\alpha(g',x')<n$ there is a \ft\ diagram
bounded by $\nf{g'}x'\nf{g'x'}^{-1}$.
Let $\Delta':=\Delta_{\nf{g},x}'$ 
and let $*$ be the basepoint of $\Delta'$.

Suppose that there is a word $w \in \nfset$
that labels two paths $p,p'$ in $\Delta'$
that start at $*$ and suppose that $p$ and $p'$
can be factored
as $p=p_1p_2p_3$ and $p'=p_1p''$
such that $p_2$ is a nonempty
edge path whose intersection with the
path $p'$ consists exactly of the initial 
vertex $i(p_2)$ and terminal vertex $t(p_2)$
of $p_2$.  
Since normal forms in $\nfset$ label simple
paths in the tree $\tree_\nfset$, they must
also label simple paths in any van Kampen diagram;
hence $i(p_2) \neq t(p_2)$.
Now we can factor the path $p'=p_1p_2'p_3'$
such that $p_2'$ is another  path in $\Delta'$
from $i(p_2)$ to $t(p_2)$.
Moreover we can write $w=w_1w_2w_3=w_1w_2'w_3'$
such that for each $i$ the word $w_i$
labels the path $p_i$ and $w_i'$ labels $p_i'$.
The images of the paths $p_1p_2$ and $p_1p_2'$
under the map $\pi_{\Delta'}$ end at the same vertex
in the Cayley graph, and so prefix closure
(and uniqueness) of the normal form set imply
that $w_2=w_2'$.
A similar argument shows that the
path $p_2'$ cannot intersect the path $p$
except at the common endpoints of $p_2$ and $p_2'$.
Hence the path $p_2p_2'^{-1}$ is a simple
loop in $\Delta'$.  
By the Jordan Curve Theorem,
this loop separates the diagram $\Delta'$ into 
two subsets.
We remove the subdiagram of $\Delta'$ contained
inside this loop, and glue the two simple
paths $p_2$ and $p_2'$; this results in a new
van Kampen diagram $\Delta''$
with the same basepoint and boundary.
Moreover, the diagram $\Delta''$ is \nformed,
and contains fewer cells than $\Delta'$;
this contradicts our choice of 
$\Delta'=\Delta_{g,x}'$ as a \nformed\ triangular
van Kampen diagram with minimal number of 2-cells.
Hence the set of edges that lie along paths in $\Delta'$
starting at $*$ and labeled by elements of $\nfset$
must form a maximal tree, 
since a pair of such paths cannot have prefixes
that diverge and then merge.


If the edge $e_{g,x}$ lies in the tree $\tree_\nfset$,
then either $\nf{gx}=\nf{g}x$ or $\nf{g}=\nf{gx}x^{-1}$.
There is a degenerate, and
hence \ft, van Kampen diagram $\Delta_{g,x}$ (= $\Delta'$)
consisting of a line segment labeled $\nf{gx}$ or $\nf{g}$,
respectively, in this case.

On the other hand, suppose that $e_{g,x}$ is not in $\tree_\nfset$.
Note that $\Delta'$ must contain a 2-cell $\sigma$
with the isolated edge $p_x$ in its boundary,
since $p_x$ is the only directed edge in the path
along $\bo\Delta'$
mapped by $\pi_{\Delta'}$ to an edge outside of $\tree_\nfset$.
If $\Delta'$ contains only one 2-cell, then 
$\Delta'$ is minimal, and hence \ft.
Suppose that $\Delta'$ contains more than
one 2-cell.
Let $v_0,v_1,...,v_t$ be the successive
vertices, and
$e_1,...e_t$ the successive edges, of the
path in $\bo\sigma \setminus p_x$ from the
initial vertex $v_0$ to the terminal vertex $v_t$ of $p_x$.
For each $0 \le i \le t$, there is a unique path $p_i$ 
from the basepoint $*$ to $v_i$
that is labeled by a word in $\nfset$.
For each $0 \le i < t$, the 
concatenated path $l_i := p_{i}e_ip_{i+1}^{-1}$
is a loop in $\Delta'$.
Let $q_i$ be the maximal common prefix of 
the pair of paths $p_i,p_{i+1}$; that
is, $p_i=q_i r_i$ and $p_{i+1}=q_i s_i$.
If $q_i$ equals one of the paths
$p_1,p_{i+1}$, then either $q_i=p_i$,
$r_i$ is a constant path, and $s_i=e_i$,
or $q_i=p_{i+1}$, $r_i=e_i^{-1}$, and $s_i$
is constant; in both cases, the loop $l_i$
follows a line segment in $\Delta'$
and returns along the same segment back to $*$,
and we let $\Delta_i$ be the degenerate
van Kampen diagram given by this line segment.
On the other hand, if $q_i$ is a proper
subpath of both $p_i$ and $p_{i+1}$, then
the  fact that normal forms from $*$
label paths in a tree shows that
the path $r_ie_is_i^{-1}$ is a simple loop in $\Delta'$.
Let $\tilde \Delta_i$ denote the 2-complex inside
this loop (including the bounding loop),
and let $\Delta_i'$ be the subdiagram 
$\Delta_i' := q_i\tilde \Delta_i$ of $\Delta'$
with the same basepoint $*$.
Again applying the fact that
normal forms label paths from $*$ that lie in a tree,
for each vertex $v$ of $\Delta_i'$, the
path in $\Delta'$ from $*$ to $v$ must lie in $\Delta_i'$;
hence $\Delta_i'$ is \nformed\ triangular.
Since the number of 2-cells in $\Delta_i'$
is at most $n-1$, by our inductive assumption
there is a \ft\ van Kampen diagram
$\Delta_i$ with the same boundary label as $\Delta_i'$.

Now the diagram $\Delta_{g,x}$
built from the disjoint union of the $\Delta_i$  and $\sigma$,
glued along the normal form paths $p_i$ and the edges $e_i$,
is a \ft\ van Kampen diagram with boundary label
$\nf{g}x\nf{gx}^{-1}$, and therefore iii) holds.
\end{proof}


\section{HNN extensions and stackability}\label{sec:hnn}


Throughout this section, let
$G=H\ast_\atob$ be the 
HNN-extension of the group $H$ with isomorphism $\atob:A\to B$ between
subgroups $A$ and $B$  and stable letter $s$. 
Assume that $H$ is finitely generated and let
$H = \langle Y \mid R_H \rangle$ be a presentation for $H$, where
$Y$ is a finite inverse closed generating system for $H$. 
Let $X:=\{s,s^{-1}\}\cup Y$,
and let $\rep:X^* \ra G$ be the canonical surjection.
Then $G$ has the presentation
\begin{equation}\label{eqn:hnn}
G = \langle X \mid R_H, \{sas^{-1}=\atob(a) \mid a \in A\}\rangle.
\end{equation}

\begin{notation}\label{britton}
{\rm   
Let $\nfset_H$ be a set of normal forms for $H$ over $Y$. 
Let $\nfset_{H/A}$ be a  
subset of $\nfset_H$ satisfying the properties that the composition 
$\nfset_{H/A} \stackrel{\rho}{\ra} H \ra H/A$ is a bijection
and $\emptyword \in \nfset_{H/A}$,
and similarly let 
$\nfset_{H/B}\subseteq\nfset_H$ be a
set of normal forms for a set of coset representatives for $B$ in $H$
that contains $\emptyword$. 
The {\em Britton set of normal forms}~\cite[p.~181]{lyndonschupp} 
for $G$ is 
\begin{align*}
\nfset_G &:=
\{
 h_1s^{\epsilon_1}h_2s^{\epsilon_2}\ldots h_ns^{\epsilon_n}h \mid 
  n\geq 0,~h\in\nfset_H~\text{ and } \epsilon_i=\pm 1\text{ for }1\leq i\leq n;\\
&~\hspace{1.75in} 
    \text{ if }\epsilon_i=+1\text{ then }h_i\in \nfset_{H/B} \\
&~\hspace{1.75in} 
    \text{ and if }\epsilon_i=-1\text{ then }h_i\in \nfset_{H/A}; \\
&~\hspace{1.75in}
    \text{and if }\epsilon_{i}=-\epsilon_{i-1}\text{ then }h_i \neq \emptyword
\};
\end{align*}
that is, 
$\nfset_G = (\nfset_{H/A} s^{-1} \cup \nfset_{H/B} s)^* \nfset_H
\setminus \cup_{\epsilon \in \{\pm 1\}} X^* s^\epsilon s^{-\epsilon}X^*$.
Given such a word 
$w=h_1s^{\epsilon_1}h_2s^{\epsilon_2}\ldots h_ns^{\epsilon_n}h$ in $\nfset_G$,
we define $\tail(w):=h_1s^{\epsilon_1}h_2s^{\epsilon_2}\ldots h_ns^{\epsilon_n}$
and $\head(w):=h$; we refer to these as the {\em tail} and {\em head}
of $w$, respectively. 
}
\end{notation}


The following is immediate from the definition of
Britton's normal forms; we will apply this in the proof 
of Theorem~\ref{thm:HNN} below.

\begin{lemma}\label{tail} 
Let $w=h_1s^{\epsilon_1}h_2s^{\epsilon_2}\ldots h_ns^{\epsilon_n}h\in\nfset$ 
and let $\tau\in\nfset$ be a tail. If either $n=0$, $h_1\neq\emptyword$,
or $h_1=\emptyword$ and $s^{\epsilon_1}$ equals the last letter of $\tau$, 
we have $\tau w\in\nfset$.
\end{lemma}


Before proceeding to Theorem~\ref{thm:HNN},
we first need a result about changing generating
sets for stackable and autostackable groups.

\begin{proposition}\label{prop:stkbladdgens}
Let $H$ be a stackable group with respect to
an inverse-closed generating set $Y$, and let 
$Z$ be an inverse-closed subset of $H$.  Then
$H$ is also stackable with respect to the generating
set $Y \cup Z$.  Moreover, if $H$ is autostackable
over $Y$, then $H$ is also autostackable over $Y \cup Z$.
\end{proposition}

\begin{proof}
Let $\ga := \ga(H,Y)$ and $\ga':=\ga(H,Y \cup Z)$ be the
Cayley graphs for $H$ over $Y$ and $Y \cup Z$, respectively.
For each $z \in Z \setminus Y$, fix a word $w_z \in Y^*$ such that
$z =_H w_z$.

Let $\ff$ be a bounded flow function for $H$
over the generating set $Y$, with associated maximal
tree $\tree$ in $\ga$, normal form set $\nfset$, and
stacking function;
$\sff = \lbl \circ \ff \circ \path: \nfset \times Y \ra Y^*$.  
The tree $\tree$ is also a maximal tree
in $\ga'$, and the associated normal forms over
$Y \cup Z$ is the same set $\nfset$.
Define the function 
$\sff': \nfset \times Y \cup Z \ra (Y \cup Z)^*$
by $\sff'(w,y) := \sff(w,y)$ for all $w \in \nfset$
and $y \in Y$, and
$\sff'(w,z) := w_z$ for all $w \in \nfset$ and 
$z \in Z \setminus Y$.  Then $\sff'$
is a stacking function for $H$ over $Y \cup Z$,
and $\ff':= \path \circ \sff' \circ \lbl$ is
a bounded flow function for 
$H$ over $Y \cup Z$.

In the case that $H$ is autostackable over $Y$,
and $\graph(\sff)$ is synchronously regular,
we have
$$
\graph(\sff') = \graph(\sff) \bigcup
(\cup_{z \in Z \setminus Y} \nfset \times \{z\} \times \{w_z\}).
$$
Since the class of synchronously regular languages is
closed under 
projection on the first coordinate,
finite direct products, and finite unions, then
$\graph(\sff')$ is also synchronously regular,
and so $H$ is autostackable over $Y \cup Z$.
\end{proof}


\begin{theorem}\label{thm:HNN}
Let $H$ be a stackable group, 
let $A,B \leq H$ be finitely generated, and let $\atob:A\to B$ be
an isomorphism.  
Then the HNN extension $G=H\ast_\atob$ is also stackable. 
\end{theorem}

\begin{proof} 
Suppose that $H$ is stackable 
with respect to
a normal form set $\nfset_H$ over an inverse-closed generating
set $Y$, and
let $\langle Y \mid R_H \rangle$ be
the finite stacking presentation 
associated to the bounded flow function for $H$ over $Y$.
Let $\nfset_{H/A}$,$\nfset_{H/B}$ be subsets of $\nfset_H$
(each containing $\emptyword$)
representing transversals of these subgroups.
Using the proof of Proposition~\ref{prop:stkbladdgens},
by possibly extending the flow function and adding
relators (of the form $z=w_z$) to the presentation, 
we may assume that the generating
set $Y$ contains a subset $Z_A$
which is an inverse-closed generating set for $A$,
as well as the subset $Z_B:=\{\atob(a) \mid a \in Z_A\}$
of $H$, which generates $B$.  (Note that this
does not affect the normal form sets
$\nfset_H$, $\nfset_{H/A}$ and $\nfset_{H/B}$.)
Using Proposition~\ref{equivalence}, for every
word $w \in \nfset_H$ and $y \in Y$, we
have a fully $\nfset_H$-triangular van Kampen
diagram $\Delta^H_{w,y}$ over the stacking presentation for $H$.

Let $\nfset_G$ be the Britton normal form set 
of Notation~\ref{britton}.
Since the set $\nfset_H$ is prefix-closed, 
then the set
$\nfset_G = (\nfset_{H/A} s^{-1} \cup \nfset_{H/B} s)^* \nfset_H
\cap (X^* \setminus \cup_{\epsilon \in \{\pm 1\}} X^* s^\epsilon s^{-\epsilon}X^*)$
is also prefix-closed.
For all $w \in X^*$, let $\nf{w}$ denote
the normal form in $\nfset_G$ of the element
of $G$ represented by $w$.
Let $\ga:=\ga(G,X)$ be the Cayley graph for
$G$ over $X:=Y \cup \{s^{\pm 1} \}$, and
let $\tree$ be the maximal tree of $\ga$
corresponding to the set $\nfset_G$.

We will show that $G$ is stackable with respect to $\nfset_G$
by showing that over the finite presentation
in Equation~(\ref{eqn:hnn})  there is a 
\ftg\ van Kampen
diagram $\Delta_{w,x}$ with boundary label $wx\nf{wx}^{-1}$
for every normal form word $w \in \nfset_G$ and generator $x \in X$;
that is, we apply Proposition~\ref{equivalence}.  
We proceed
via several cases depending upon $w$ and $x$; in each case, we have that 
\ftg\ van Kampen diagrams have been 
constructed for the prior cases.
 
\medskip

\noindent{\em Case 1: Suppose that $x\in Y$.}
Let $\Delta^H_{\head(w),x}$ be the 
fully $\nfset_H$-triangular van Kampen diagram associated to $\head(w)$ and $x$.
By Lemma~\ref{tail}, for all $v \in \nfset_H$ 
we have $\tail(w)v \in \nfset_G$.  
Thus the diagram $\tail(w)\Delta^H_{\head(w),x}$
is a fully $\nfset_G$-triangular  van Kampen diagram in $G$,
with boundary label $wx\nf{wx}^{-1}$. 

\medskip

\noindent{\em Case 2: Suppose that $x=s^{\pm 1}$.} 
The two cases $x=s$ and 
$x=s^{-1}$ are analogous, so we assume that 
$x=s$ and leave the case when $x=s^{-1}$ to the reader. 

Let $SL_B \subset Z_B^*$ denote the set of shortlex normal
forms for $B$ over the generating set $Z_B$ with respect
to some total ordering of $Z_B$.
The element of $H$ represented by $\head(w)$ can be
written $\head(w)=_H uv$ for a unique $u \in \nfset_{H/B}$
and $v \in SL_B$. 
 
We proceed by induction on the length of $v$.
If $v = \emptyword$, then either $wx=ws \in \nfset_G$
or else $w$ ends with $s^{-1}$; in either case,
there is a degenerate
van Kampen diagram $\Delta_{w,x}$
with boundary word $ws\nf{ws}^{-1}$.

Next suppose that $l(v)>0$ and that
we have a fully $\nfset_G$-triangular van Kampen
diagram for $\nf{\tail(w)uv'}s\nf{\tail(w)uv's}^{-1}$ 
for all $v' \in SL_B$ with
$l(v')<l(v)$.
Write $v=v'b$ with $v' \in SL_B$ and $b \in Z_B$.
Let $a \in Z_A$ be the letter 
satisfying $\atob(a)=b$.


The normal form $\nf{wb^{-1}}$ is
$w' := \tail(w)\nf{uv'}$; then since
$\head(w') =_H uv'$ and $l(v')<l(v)$, 
our inductive assumption implies that there is a 
fully $\nfset_G$-triangular van Kampen diagram
$\Delta_2$ 
with boundary word $\nf{wb^{-1}}s\nf{wb^{-1}s}^{-1}$.

\begin{figure}\centering
\caption{The diagram $\Delta_{w,s}$.}
\begin{tikzpicture}[scale=2]
\draw [very thick,decoration={markings, mark=at position .8 with {\arrow{>}}}, postaction={decorate}] (0,0) -- node[fill=white, sloped, pos=.5]
  {{\small $\nf{wb^{-1}sa}$}} (20:3cm);
 \draw [very thick,decoration={markings, mark=at position .9 with {\arrow{>}}}, postaction={decorate}] (0,0) -- node[fill=white, sloped, pos=.6] {{\small $\nf{wb^{-1}s}$}}   (6:2.4cm);
\draw [very thick,decoration={markings, mark=at position .9 with {\arrow{>}}}, postaction={decorate}] (0,0) -- node[fill=white, sloped,  pos=.6] {{\small $\nf{wb^{-1}}$}}    (-6:2.4cm);
 \draw [very thick,decoration={markings, mark=at position .8 with {\arrow{>}}}, postaction={decorate}] (0,0) -- node[fill=white, sloped, pos=.5] {{\small $w$}}  (-20:3cm);
 \draw [very thick,decoration={markings, mark=at position .7 with {\arrow{>}}}, postaction={decorate}]  (6:2.4cm) -- node[fill=white, pos=.4] {{\small $a$}}  (20:3cm);
 \draw [very thick,decoration={markings, mark=at position .8 with {\arrow{>}}}, postaction={decorate}] (-6:2.4cm) -- node[fill=white, pos=.4] {{\small $s$}}  (6:2.4cm);
 \draw [very thick,decoration={markings, mark=at position .8 with {\arrow{>}}}, postaction={decorate}] (-20:3cm) -- node[fill=white, pos=.5] {{\small $b^{-1}$}} (-6:2.4cm);  
 \draw [very thick,decoration={markings, mark=at position .8 with {\arrow{>}}}, postaction={decorate}] (-20:3cm)-- node[fill=white, pos=.5] {{\small $s$}} (20:3cm);
\node at (2.1,0) {$\Delta_2$};
\node at (2,-.5) {$\Delta_1$};
\node at (2,.5) {$\Delta_3$};
\end{tikzpicture}
\end{figure}
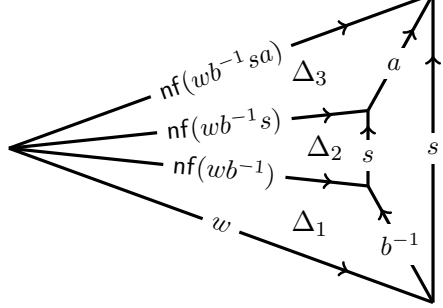

Applying Case 1, there are
fully $\nfset_G$-triangular diagrams
$\Delta_1=\tail(w)\Delta^H_{\head(w),b^{-1}}$
and 
$\Delta_3=\tail(wb^{-1}s)\Delta^H_{\head(wb^{-1}s),a}$ (see Figure 1). Gluing the diagrams $\Delta_1,\Delta_2$ along
their (simple) boundary paths $\nf{wb^{-1}}$,
and gluing the resulting diagram with $\Delta_3$
along their $\nf{wb^{-1}s}$ paths, results in
the subdiagram $\tilde \Delta$ of Figure 1
with boundary label $wb^{-1}sa$.
We glue a single 2-cell with boundary label
$b^{-1}sas^{-1}$ to $\tilde \Delta$ along their
$b^{-1}sa$ boundary paths; this yields a
fully $\nfset_G$-triangular van Kampen diagram 
$\Delta_{w,x}$ with boundary label $wx\nf{wx}^{-1}$.
\end{proof}

\begin{remark}\label{rmk:hnn}
{\rm 
We record here, for later use, the result of Theorem~\ref{thm:HNN}
in terms of stacking maps.
Suppose that $\sff_H:\nfset_H \times Y \ra Y^*$ is the
stacking map associated to the bounded flow
function on $H$ in Theorem~\ref{thm:HNN}.
For each
word $w \in \nfset_G$
define $\trans_A(w)$ and $\subg_A(w)$ to 
be the unique elements of the transversal
$\nfset_{H/A}$
and subgroup shortlex representatives $SL_{A}$,
respectively,
such that $\head(w) =_H \trans_A(w)\subg_A(w)$,
and let $\last(\subg_A(w))$ denote the
last letter (in $Z_{A}$) of the word $\subg_A(w)$.
Similarly define $\trans_B(w)$, $\subg_B(w)$,
and $\last(\subg_B(w))$.
The stacking map for $G$ over $X$ is the function
$\sff_G: \nfset_G \times X \ra X^*$ defined
for
all $w \in \nfset_G$ and $x \in Y$ by:
$$
\sff_G(w,x):=
\begin{cases}
\sff_H(\head(w),x) 
    & \text{if } x \in Y \\
x 
    & \text{if } x=s \text{ and } \head(w)=\trans_B(w)  \\
    & \text{or } x=s^{-1} \text{ and } \head(w)=\trans_A(w) \\
\last(\subg_B(w))^{-1} s \atob^{-1}(\last(\subg_B(w)))
    & \text{if } x=s \text{ and } \head(w) \neq \trans_B(w)  \\
\last(\subg_A(w))^{-1} s^{-1} \atob(\last(\subg_A(w)))
    & \text{if } x=s^{-1} \text{ and } \head(w) \neq \trans_A(w).  \\
\end{cases}
$$
}
\end{remark}

In the case that the graph of the stacking
map $\sff_H$ for $H$ is decidable or synchronously regular,
the proof of
Theorem~\ref{thm:HNN} does not 
show that the same must hold for the graph of
the stacking function $\sff_G$.  However, in
many special cases, this does hold.  

\begin{corollary}\label{cor:autohnn}
Let $H$ be an autostackable [respectively, algorithmically
stackable] group 
over an inverse-closed generating set $Y$. 
Let $A \leq H$ be 
generated by a finite inverse-closed set $Z \subseteq Y$
with shortlex normal form set $SL_A$
(with respect to some total ordering of $Z$),
and let $\atob:A\to H$ be a monomorphism
with $\atob(Z) \subseteq Y$.  
Suppose further that there are 
regular [respectively, decidable]
subsets $\nfset_{H/A},\nfset_{H/\atob(A)} \subseteq \nfset_H$,
each containing $\emptyword$, representing
transversals of these subgroups, and that
for each $z \in Z$ and $\tilde z \in \atob(Z)$, the sets 
\begin{eqnarray*}
L_z & := & \{w \in \nfset_H \mid w=_H \trans_A(w)\subg_A(w)
\text{ for some } \\
&& \ \hspace{.6in} \trans_A(w) \in \nfset_{H/A}
\text{ and } \subg_A(w) \in SL_{A} \cap Z^*z\} \text{ and} \\
L_{\tilde z}' & := & \{w \in \nfset_H \mid 
    w=_H \trans_{\atob(A)}(w)\subg_{\atob(A)}(w)
\text{ for some } \\
&& \ \hspace{.6in} \trans_{\atob(A)}(w) \in \nfset_{H/\atob(A)}
\text{ and } \subg_{\atob(A)}(w) \in \atob(SL_{A}) \cap \atob(Z)^*\tilde z\}
\end{eqnarray*}
are also regular [respectively, decidable].
Then the HNN extension $G=H\ast_\atob$ is autostackable
[respectively, algorithmically stackable]. 
\end{corollary}

\begin{proof}
We give the proof in the autostackable case; the
algorithmically stackable proof is similar.
In the notation of the proof of Theorem~\ref{thm:HNN}, 
let $s \in G$ be the stable letter of the HNN extension, and 
let $X=Y \cup \{s^{\pm 1}\}$.

Let $\sff_H$ be the the stacking map for the
autostackable structure for $H$,
and let $\sff_G$ be the stacking map for
$G$ from Remark~\ref{rmk:hnn}.
Now $\graph(\sff_H)$ is \syreg.
Let $\proj_1,\proj_3:(X^*)^3 \ra X^*$ be the
projection maps on the first and third coordinates.
Note that the normal form set for $H$ is 
$\nfset_H = \proj_1(\graph(\sff_H))$,
and since \syreg\ languages are closed
under projections, the set $\nfset_H$ is regular.
The normal form set for $G$ has the form
$\nfset_G=\Tail \cdot \nfset_H$ where
$$
\Tail:=(\nfset_{H/A}s^{-1} \cup \nfset_{H/\atob(A)}s)^* 
\cap (X^* \setminus 
\cup_{\epsilon \in \{\pm 1\}}X^*s^\epsilon s^{-\epsilon}X^*),
$$ 
and so $\nfset_G$ is built from regular languages using intersection,
union, complementation, concatenation 
and Kleene star. 
(See Section~\ref{sec:notation} for properties of
regular and \syreg\ languages.)
Hence $\nfset_G$ also is a regular language.

The graph of $\sff_G$ can be written
\begin{eqnarray*}
\graph(\sff_G) &=& 
 \big( \cup_{y \in Y,v \in \proj_3(\graph(\sff_H))}
  L_1
  \times \{y\} \times \{v\}  \big) \\
&&\cup
\big(\cup_{\epsilon \in \{\pm 1\}}
  L_{2,\epsilon} \times \{s^{\epsilon}\} \times \{s^{\epsilon}\} \big) \\
&&\cup
\big(\cup_{\tilde z \in \atob(Z)}
  L_{3,\tilde z}'\times \{s\} \times \{{\tilde z}^{-1}s\atob^{-1}(\tilde z)\} \big)\\ 
&&\cup
\big(\cup_{z \in Z}
  L_{3,z}\times \{s^{-1}\} \times \{z^{-1}s^{-1}\atob(z)\} \big)
\end{eqnarray*} 
where
\begin{eqnarray*}
L_1 &=&  \Tail \cdot \proj_1(\graph(\sff_H) \cap (X^* \times \{y\} \times \{v\})),\\
L_{2,1} &=&
  \Tail \cdot \nfset_{H/\atob(A)},  \\
L_{2,-1} &=&
  \Tail \cdot \nfset_{H/A}, \\
L_{3,z} &=& \Tail \cdot L_{z}, \text{   and} \\
L_{3,\tilde z}' &=& \Tail \cdot L_{\tilde z}'.
\end{eqnarray*} 
Again closure properties of regular
and \syreg\ languages show that
each of these languages is regular, and hence
so is $\graph(\sff_G)$.
\end{proof}


In the algorithmically stackable case,
Corollary~\ref{cor:autohnn} can be rephrased 
in terms of solvability of the subgroup membership problem.

\begin{corollary}\label{cor:alghnn}
Let $H$ be an algorithmically stackable group, 
let $A,B \leq H$ be finitely generated, and let $\atob:A\to B$ be
an isomorphism.  
Suppose further that the subgroup membership problem
is decidable for the subgroups $A$ and $B$ in $H$.
Then the HNN extension $G=H\ast_\atob$ is also 
algorithmically stackable. 
\end{corollary}

\begin{proof}
From the proof of Theorem~\ref{thm:HNN},
there is a finite inverse-closed
generating set $Y$ for $H$ containing
a finite inverse-closed set $Z$ of generators
for $A$ as well as the generators $\atob(Z)$
of $B=\atob(A)$, and there is a stackable structure
for $G=H\ast_\atob$ over $X=Y \cup \{s^{\pm 1}\}$
with Britton normal form set $\nfset_G$, such that
the associate stacking map is given in
Remark~\ref{rmk:hnn}.  From Corollary~\ref{cor:autohnn},
then, it suffices to show that there are
decidable transversals 
 $\nfset_{H/A},\nfset_{H/B} \subseteq \nfset_H$,
each containing $\emptyword$, such that the
languages $L_z$ and $L_{\tilde z}'$ of 
Corollary~\ref{cor:autohnn} are also decidable.

Let $<_{SL}$ denote the shortlex ordering 
on $Y^*$ corresponding to a total ordering of $Y$.
For each coset $hA$ of $H/A$, let 
$\tau_{hA}$ denote the shortlex least word in
$\nfset_H$ representing an element of $hA$,
and let 
$$
\nfset_{H/A} =\{\tau_{hA} \mid hA \in H/A\}.
$$
Note that the empty word $\emptyword$ is
an element of $\nfset_{H/A}$.

In order to determine whether a given
word $w \in Y^*$ lies in $\nfset_{H/A}$,
first use decidability to determine whether
$w \in \nfset_H$.  If not, then 
(halt and output) $w \notin \nfset_{H/A}$;
if so, we 
next enumerate the finite set $S$
of elements of $Y^*$
satisfying $v <_{SL} w$ for all $v \in S$.
For each word $v \in S$, 
use decidability to determine whether
$v \in \nfset_H$ and
use the solution of the
subgroup membership problem to determine
whether $v^{-1}w \in A$.
If there is a word $v \in S$ with $v \in \nfset_H$
and $v^{-1}w \in A$, then $w \notin \nfset_{H/A}$;
and if there is no such word in $S$, then
$w \in \nfset_{H/A}$.
Hence $\nfset_{H/A}$ is decidable.

Next suppose that $z \in Z$ and consider the set of Corollary \ref{cor:autohnn}
$$
L_z := \{u \in \nfset_H \mid u=_H \tau\sigma
\text{ for some }
\tau \in \nfset_{H/A}
\text{ and } \sigma \in SL_{A} \cap Z^*z\}
$$
where (as before) $SL_A$ is the set of shortlex
normal forms for $A$ over $Z$.
The algorithm to determine whether a
given word $w$ over $Y$ lies in $L_z$ 
also begins by using
decidability to determine whether
$w \in \nfset_H$, and if not,
halts with $w \notin L_z$.
If $w \in \nfset_H$,
then we repeat the algorithm in
the previous paragraph to compute the
word $\tau \in \nfset_{H/A}$ satisfying
$\tau=\tau_{wA}$; that is, $\tau A=wA$ and so
$\tau^{-1}w \in A$, and moreover $\tau \le_{SL} w$
is the shortlex least word with this property.
Next enumerate all words $y_0,y_1,y_2,...$
over $Z$ in increasing shortlex order.
Now since the word $\tau^{-1}w$ represents an
element of $A$, we have $\tau^{-1}w =_H y_j$ for
some indices $j$; we can use the solution
of the word problem from the algorithmically
stackable structure on $H$ to determine the
first index $i$ for which $\tau^{-1}w =_H y_i$.
Then $\sigma:=y_i \in \SL_A$ and $w=_H \tau\sigma$.
Now $w \in L_z$ iff $y_i$ ends with the letter $z$.
Thus $L_z$ is also decidable.

A similar argument shows that the set 
$$
\nfset_{H/B}=\{\tau_{hB}' \mid hB \in H/B\},
$$ where $\tau_{hB}'$ denotes the shortlex least
word in $\nfset_H$ representing an element of $hB$, is decidable
and contains $\emptyword$,
and for each $\tilde z \in \atob(Z)$
the set
$$
L_{\tilde z}' :=  \{w \in \nfset_H \mid 
    w=_H \tau\sigma
\text{ for some } \tau \in \nfset_{H/\atob(A)}
\text{ and } \sigma \in \atob(SL_{A}) \cap \atob(Z)^*\tilde z\}
$$ is also decidable.
\end{proof}

\bigskip


\section{Applications and Dehn functions
}\label{sec:stkapplications}


In this section we give three applications of 
Theorem~\ref{thm:HNN} and 
Corollaries~\ref{cor:autohnn}~and~\ref{cor:alghnn},
that give information on the Dehn functions of
stackable, algorithmically stackable, and
autostackable groups.

\bigskip


\subsection{Stackable versus autostackable}\label{subsec:unsolvablewp}


In the first application, we show that stackability
and autostackability are not the same property,
and that the class of stackable groups contains
groups whose Dehn function is not computable.

\begin{theorem}\label{thm:stknotauto}
There exists a stackable group with unsolvable word problem,
and hence stackability does not imply algorithmic stackability.
\end{theorem}

\begin{proof}
Let $C=\langle Y \mid R \rangle$ be a finitely
presented group with unsolvable word problem.
Let $Y'$ be a copy of $Y$, and let
$H=F(Y) \times F(Y')$ be a direct copy of
the free groups generated by $Y$ and $Y'$.
Also let $\rho:F(Y) \ra C$ and $\rho':F(Y') \ra C$
be the quotient maps.  Let $A$ be the
{\em Mihailova subgroup}
$$
A = \{(h,h') \in H \mid \rho(h)=\rho'(h')\}
$$
associated to $C$.
Mihailova~\cite{mihailova} showed that the
subgroup membership problem for $A$ in $H$ 
is not decidable; that is, there does not
exist an algorithm that
upon input of a word $w$
in the generating set $(Y \cup Y')^{\pm 1}$
of $H$,  can determine whether $w$
represents an element of the subgroup $A$.
The group $A$ is finitely generated
(see for example the paper of Bogopolski
and Ventura~\cite{bv} for a discussion and recursive
presentation for this group); let $Z$ be
a finite generating set for $A$.  

To construct an HNN extension from this
data, we let $\atob:A \ra A$ be the identity
function on $A$, and let $G=H*_\atob$.
Let $\widetilde Y:=Y \cup Y' \cup Z$, and 
for each $z \in Z$, let $w_z \in ((Y \cup Y')^{\pm 1})^*$
be a word satisfying $z =_H w_z$.  Then
\[\begin{array}{lr}
G=\langle \widetilde Y \cup \{s\} \mid & 
[y,y']=\emptyword \text{ for all }y \in Y, y' \in Y',~ \text{ and } \\
& z=w_z \text{ and } szs^{-1}=z \text{ for all }z \in Z \rangle.
\end{array}\]
Since the group $H$ is a direct product
of free groups, $H$ is a stackable group.
Now the proof of Theorem~\ref{thm:HNN}
shows that $G$ is a stackable group, with a
stackable structure over the generating set 
$(\widetilde Y \cup \{s\})^{\pm 1}$ yielding the above
as the stacking presentation.

If the word problem for $G$ were to have a solution,
then upon input of any word $sws^{-1}$
with $w \in (Y \cup Y')^{\pm 1~*} \subseteq \widetilde Y^{\pm 1~*}$,
the word problem algorithm can determine whether or not
$sws^{-1}w^{-1}=_G \groupid$.  However,
$sws^{-1}w^{-1}=_G \groupid$ if and only if $w$
represents an element of the subgroup $A$
in the domain of $\atob$.  Hence this solves
subgroup membership as well, giving a contradiction.
Since $G$ does not have solvable word
problem, this stackable group $G$ cannot
be algorithmically stackable.
\end{proof}

\bigskip


\subsection{Dehn functions for algorithmic stackability:  Hydra groups}\label{subsec:algdehn}


In our second application,
we show that
Dehn functions of algorithmically stackable groups can
be extremely large.

\begin{theorem}\label{thm:hydra}
The class of algorithmically stackable groups includes 
groups with Dehn functions in each level of
the Grzegorczyk hierarchy of primitive 
recursive functions.
\end{theorem}

\begin{proof}
Dison and Riley~\cite{disonriley} defined a family of 
groups $\Gamma_k$ (for $k \ge 2$),
built by HNN extensions, and showed that 
the Dehn function of $\Gamma_k$ is 
equivalent to the $k$-th Ackerman function.
In particular, 
for each integer $k \ge 2$, the group
$\Gamma_k=G_k *_{\atob_k}$ 
is an HNN extension of a free-by-cyclic group
$$
G_k=\langle a_1,...,a_k,t \mid ta_1t^{-1}=a_1,~ ta_it^{-1}=a_ia_{i-1}~(i>1)\rangle
$$
(known as a hydra group)
with respect to the identity map
$\atob_K:H_k \ra H_k$ on the finitely
generated (rank $k$ free) subgroup 
$H_k=\langle a_1t^{-1},...,a_kt^{-1}\rangle$.
Since the class of algorithmically
stackable groups is closed
under extension~\cite{bhj},
the group $G_k$ is algorithmically stackable.  
Theorem~\ref{thm:HNN} shows that $\Gamma_k$ is
also stackable.
Dison, Einstein and Riley~\cite[Theorem~3]{der} have
shown that for the subgroup
$H_k$ of $G_k$, the subgroup membership
problem is decidable.  Then Corollary~\ref{cor:alghnn}
shows that $\Gamma_k$ is also algorithmically stackable.
\end{proof}

\bigskip


\subsection{Dehn functions for autostackability: Baumslag's nonmetabelian group}\label{subsec:autodehn}









In this third application we consider
Baumslag's nonmetabelian group, also known as the
Baumslag-Gersten group, which is presented by
$$
G = \langle a,s \mid (sas^{-1})a(sa^{-1}s^{-1})=a^2 \rangle
  = \langle a,t,s \mid tat^{-1}=a^2,~sas^{-1}=t\rangle.
$$
This group can be realized as an HNN extension
$G = H *_\atob$
where $H$ is the Baumslag-Solitar group
$H = BS(1,2) = \langle a,t \mid tat^{-1}=a^2\rangle$
and $\atob:\langle a \rangle \ra \langle t \rangle$
is the map given by $\atob(a)=t$.  
The group $H$ is autostackable~\cite{MarkSusan,bhh}, and so
Theorem~\ref{thm:HNN} shows that $G$ is stackable.  
We strengthen this result to show the following.

\begin{theorem}\label{thm:nonmet}
Baumslag's nonmetabelian group 
$\langle a,s \mid (sas^{-1})a(sa^{-1}s^{-1})=a^2 \rangle$
is autostackable.
\end{theorem}

\begin{proof}
The Baumslag-Solitar group
$H=\langle a,t \mid tat^{-1}=a^2 \rangle$
has a finite complete rewriting system on the
generating set $Y=\{a^{\pm 1},t^{\pm 1}\}$
given by
$$
\{a^\epsilon a^{-\epsilon} \ra \emptyword,~
t^\epsilon t^{-\epsilon} \ra \emptyword,~
a^2t \ra ta, a^{-1}t \ra ata^{-1},~
a^\epsilon t^{-1} \ra t^{-1}a^{2\epsilon} 
\mid \epsilon \in \{\pm 1\}~\},
$$
and hence is autostackable~\cite{bhh}.
The normal form set of this autostackable
structure is the regular language
\begin{eqnarray*}
\nfset_H &=& [(\{1,a\}\cdot t,t^{-1})^* 
  \setminus \cup_{\epsilon \in \{\pm 1\}} Y^*t^\epsilon t^{-\epsilon}Y^*]
  (a^* \cup (a^{-1})^*) \\
&=& [(t^{-1})^* \cup ((t^{-1})^*at \cup \emptyword)(\{\emptyword,a\}\cdot t)^*]
  (a^* \cup (a^{-1})^*).
\end{eqnarray*}
(That is, $H$ is an HNN extension of the
infinite cyclic group $\langle a \rangle$
by the monomorphism $\langle a \rangle \ra \langle a \rangle$
defined by $a \mapsto a^2$, and
$\nfset_H$ is the associated set of Britton normal
forms.)

Let $A=\langle a \rangle$ and $B=\langle t \rangle$,
subgroups of $H$, and let $\atob:A \ra B$ be
the map $\atob(a)=t$, so that $G=H \ast_{\atob}$.
Then the generating set $Z:=\{a^{\pm 1}\}$ for
$A$ and its image $\atob(Z)=\{t^{\pm 1}\}$
are both subsets of the generating set $Y$ of $H$.

By Corollary~\ref{cor:autohnn}, it now suffices to
show that there are regular transversals for
these subgroups such that each of the
languages $L_z$ and $L_{\tilde z}'$ is
regular.

Define
\begin{eqnarray*}
\nfset_{H/A} &:=& (t^{-1})^*
\cup  ((t^{-1})^*at \cup \emptyword)(\{\emptyword,a\}\cdot t)^*
\text{   \ \ \ and} \\
\nfset_{H/B} &:=& a^* \cup (a^{-1})^* \cup 
t^{-1}(t^{-1})^* \cdot (a(a^2)^* \cup a^{-1}(a^{-2})^*).
\end{eqnarray*}
Then $\nfset_{H/A}$ and $\nfset_{H/B}$ are
 subsets of $\nfset_H$ (each containing $\emptyword$)
that are transversals
for $A$ and $B$ in $H$, respectively.
  We note that
$\nfset_{H/A}$ and $\nfset_{H/B}$ are built from finite
sets using unions, concatentations, and Kleene star,
and so both of these sets are also regular languages.

The set of
shortlex normal forms for elements of the subgroup
$A$ over the generating set $Z$ is
$SL_A = a^* \cup (a^{-1})^*$, and similarly 
the shortlext normal forms for $B$ over $\atob(Z)$ is
$SL_B = t^* \cup (t^{-1})^*$.  

Let $z \in Z$.  Then $z=a^\epsilon$ for some
$\epsilon \in \{\pm 1\}$.  Since
$\nfset_H=\nfset_{H/A}SL_A$, the language
$L_{a^\epsilon}$ satisfies
\begin{eqnarray*}
L_{a^\epsilon} &=& \{w \in \nfset_H \mid w=_H \trans_A(w)\subg_A(w)
\text{ for some }
\trans_A(w) \in \nfset_{H/A} \\
&& ~\hspace{1in} \text{ and } \subg_A(w) =a^{\epsilon i} \text{ with } i>0\} \\
&=& \nfset_H \cap Y^*a^{\epsilon}.
\end{eqnarray*}
Then $L_{a^\epsilon}$ is an intersection of regular
languages, and hence is also regular.

Next, for $\tilde z \in \atob(Z)$,
we have $\tilde z=t^{\epsilon}$ with $\epsilon \in \{\pm 1\}$.
Suppose first that $\epsilon=1$.
Then
\begin{eqnarray*}
L_{t}' &=&  \{w \in \nfset_H \mid 
    w=_H \trans_B(w)\subg_B(w)
\text{ for some } \trans_B(w) \in \nfset_{H/\atob(A)} \\
&& ~\hspace{1in} \text{ and } \subg(w)=t^k \text{ with } k>0\}.
\end{eqnarray*}
Suppose that $w \in \nfset_H \cap Y^*tY^*$;
that is, 
$w \in ((t^{-1})^*at \cup t)(\{\emptyword,a\}\cdot t)^*]
  (a^* \cup (a^{-1})^*)$.
Then either 
$w=t^{-i}ata^{\epsilon_1} \cdots ta^{\epsilon_k}ta^\ell$,
or
$w=ta^{\epsilon_1} \cdots ta^{\epsilon_k}ta^\ell$
for some $i,k \ge 0$, $\epsilon_i \in \{0,1\}$
and $\ell \in \Z$.  
Hence either
$\trans_B(w)=t^{-i}a^{1+2\epsilon_1+\cdots+2^k\epsilon_k+2^{k+1}\ell}$
or 
$\trans_B(w)=a^{2\epsilon_1+\cdots+2^k\epsilon_k+2^{k+1}\ell}$
(respectively),
and $\subg_B(w)=t^{k+1}$.  Since $k \ge 0$, then
$\last(\subg(w))=t$, and so
$w \in L_t'$.  Hence $L_t' \supseteq \nfset_H \cap Y^*tY^*$.
On the other hand, for any $v \in L_t'$,
we have $v=\nf{a^jt^k}$ or $v=\nf{t^{-i}a^{2j+1}t^k}$
for some $i \ge 0$, $j \in \Z$, and $k>0$.
Applying the rules of the rewriting system above,
then the normal form $v$ must contain the letter $t$.
That is,
$$
L_t'=\nfset_H \cap Y^*tY^*
$$
and therefore this set is a regular language.

Finally we consider the set
\begin{eqnarray*}
L_{t^{-1}}' &=&  \{w \in \nfset_H \mid 
    w=_H \trans_B(w)\subg_B(w)
\text{ for some } \trans_B(w) \in \nfset_{H/\atob(A)}\\
&& ~\hspace{1in} \text{ and } \subg(w)=t^{-i} \text{ with } i>0\}.
\end{eqnarray*}
In this case we have 
$$
L_{t^{-1}}'=\nfset_H \setminus (\nfset_{H/B} \cup L_t')
=t^{-1}(t^{-1})^*((a^2)^* \cup (a^{-2})^*),
$$
and so $L_{t^{-1}}'$ is also regular.

Corollary~\ref{cor:autohnn} now shows that $G$ is
autostackable.
\end{proof}

The following Corollary is now immediate from 
Theorem~\ref{thm:nonmet} and Platonov's proof 
that the Dehn function of Baumslag's nonmetabelian
group is not elementary~\cite{platonov}.

\begin{corollary}\label{cor:autodehn}
The class of autostackable groups 
includes groups with nonelementary Dehn functions.
\end{corollary}







\section{Autostackable metabelian groups}\label{sec:metabelian}


In this section we consider an infinite family of
nonconstructible metabelian groups.
Let $p \in \{n \in \Z \mid n \ge 2\} \cup \{\infty\}$, and 
let 
$$
G_p=\langle a,s,t\mid a^p=1, [a^t,a]=1,a^s=a^ta, [s,t]=1\rangle,
$$
where the case $p=\infty$ means that no relation $a^p=1$ occurs. 
The group $G_\infty$ is Baumslag's metabelian group,
which is introduced in~\cite{Baumslag1},
and for $p<\infty$, the torsion analog $G_p$ of
Baumslag's metabelian group is the Diestel-Leader
group $\Gamma_3(p)$ (which is also metabelian).
Our objective in this section is to show in Theorem~\ref{thm:metabelian}
that $G_\infty$ is 
algorithmically stackable and
the groups $G_p$ for $p<\infty$ are autostackable.

We begin with a description of the subgroup structure 
of $G_p$, following~\cite{Baumslag1}.
Let  $H_p $ be the subgroup of $G_p$
generated by $Y=\{a^{\pm 1},t^{\pm 1}\}$.
In his paper~\cite{Baumslag1}, 
Baumslag showed that
a consequence of the relations in the presentation
above of $G_\infty$ is that
$[a^{t^i},a^{t^j}]=1$ for all $i,j \in \Z$.  Moreover 
\begin{eqnarray*}
H_p &=& 
\langle a,t \mid [a^{t^i},a^{t^j}]=1 \text{ for all } i,j \in \Z , a^p=1\rangle\\
& = & \big( \bigoplus_{i \in \Z} \langle t^iat^{-i}\rangle \big) 
  \rtimes \langle t \rangle,
\end{eqnarray*}
where 
$\langle t^iat^{-i}\rangle$ is
isomorphic to $\Z_p$ for each $i$,
$\langle t \rangle \cong \Z$, and $t$ acts 
on $\bigoplus_{i \in \Z} \langle t^iat^{-i}\rangle$
conjugating the $i$-th summand to the 
$(i+1)$-th summand; that is, $H_p$
is the (restricted) wreath product $H_p = \Z_p \wr \Z$.
(In the case that $p=2$, the group $H_2$ is also
known as the lamplighter group.)
Let $\atob:H_p \ra \langle a^ta,t \rangle \le H_p$
be the map defined by $\atob(a)=a^ta$ and $\atob(t)=t$;
then the group $G_p$ is the HNN extension
$G_p = H_p *_{\atob_p}$, and
the generator $s$ of $G_p$ is the corresponding stable letter.

The crucial difference with Theorem~\ref{thm:HNN}
is that in this case the group $H_p$ 
is not finitely presentable, 
so  $H_p$ cannot have a stackable structure.
Despite this, there are some analogies between
the proofs of 
Theorems~\ref{thm:HNN} and~\ref{thm:metabelian}; 
in particular, the Britton set of normal forms are
used for the HNN extensions in both.

In order to describe the normal form set
for $G_p$, and to streamline other parts
of the proof of Theorem~\ref{thm:metabelian},
we also make use of an another way to view 
the elements of this group.
Using the isomorphism $\ism$ between
$\big( \bigoplus_{i \in \Z} \langle t^iat^{-i}\rangle \big)$
and $\Z_p[x,\frac{1}{x}]$ given by
$\ism((t^{i_1}a^{\beta_1}t^{-i_1}) \cdots (t^{i_n}a^{\beta_n}t^{-i_n})) 
:= \beta_1x^{i_1} + \cdots + \beta_nx^{i_n}$,
there are isomorphisms
\begin{eqnarray*}
\ism: H_p & \ra & \widehat H_p := 
\Z_p[x,\frac{1}{x}] \rtimes \langle \hat t \rangle 
\hspace{.2in} \text{ and} \\ 
\ism: G_p & \ra & \widehat G_p := 
\big( \Z_p[x,\frac{1}{x}] \rtimes \langle \hat t \rangle \big) *_{\hat \atob}
\end{eqnarray*}
(with $\ism(t):=\hat t$ and $\ism(s):=\hat s$)
where the conjugation action of
$\hat t$ on $\Z_p[x,\frac{1}{x}]$ is multiplication by $x$,
and the map 
$\hat \atob: H_p \ra \langle 1+x \rangle \rtimes \langle \hat t \rangle$ 
is defined by $\hat \atob(x^0):=1+x$ and $\hat \atob(\hat t):= \hat t$;
that is, the conjugation action by the stable letter $\hat s$
on $\widehat H_p$ is given by multiplication by $1+x$
on the $\Z_p[x,\frac{1}{x}]$ subgroup and fixes $\hat t$.

\begin{theorem}\label{thm:metabelian}
Baumslag's metabelian group 
$G_\infty=\langle a,s,t\mid a^s=a^ta, [a^t,a]=1,[s,t]=1\rangle$
is algorithmically stackable, 
and the Diestel-Leader
torsion analogs
$G_p=\langle a,s,t\mid a^s=a^ta, [a^t,a]=1,[s,t]=1,a^p=1\rangle$
with $p \ge 2$ are 
autostackable.
\end{theorem}

\begin{proof}
Although we cannot directly apply Theorem~\ref{thm:HNN}, 
we will use roughly the same ingredients in order 
to build a stacking system for the group $G_p$.
Consider the inverse closed generating set
$$
X = \{a^{\pm 1},s^{\pm 1},t^{\pm 1}\}
$$
for $G_p$.

\bigskip

\noindent{\it Step I. Normal forms and notation:}

\smallskip

In order to build a set of normal forms for the group $H_p$,
we use the isomorphism $H_p \cong \widehat H_p$,
and note that an arbitrary element of $\widehat H_p$
can be written uniquely in the form
$p(x)\hat t^m$ where 
$p(x) \in \Z_p[x,\frac{1}{x}]$
and $m \in \Z$.
Let
\begin{eqnarray*}
\nfset_{H_p} &:=& \{t^m \mid m \in \Z\} \cup \\
&&\{t^{r}a^{\alpha_{r}}ta^{\alpha_{r+1}}\ldots ta^{\alpha_{l}}t^{-l+m}
\mid 
r,l,m\in\Z,r\leq l,\alpha_i\in\Z_p\text{ for }r\leq i\leq l, \\
&& ~\hspace{2in} \text{and } \alpha_r,\alpha_l \neq 0\};
\end{eqnarray*}
here (and throughout this proof)
we write $\alpha_i \in \Z_p$ to mean that
$\alpha_i \in \Z$ in the case that $p=\infty$,
and $\alpha_i \in \{0,1,...,p-1\}$ if $p$ is finite.
Then the restriction of the map $\ism$ to
the set $\nfset_{H_p}$
gives a bijection 
$\bj:\nfset_{H_p} \ra \widehat H_p$ defined by
$\bj(t^m) := \hat t^m$ and
$$
\bj(t^{r}a^{\alpha_{r}}ta^{\alpha_{r+1}}\ldots ta^{\alpha_{l}}t^{-l+m})
:=
(\alpha_rx^r+ \alpha_{r+1}x^{r+1} + \cdots +\alpha_lx^l)\hat t^m,
$$ 
where the integers $r$ and $l$ are the lowest and
highest degrees in the polynomial in the
$\Z_p[x,\frac{1}{x}]$ subgroup, respectively, and
hence $\nfset_{H_p}$ is a set of normal forms
for $H_p$.

The HNN extension $G_p=H_p *_\atob$ is strictly ascending,
in that the isomorphism $\atob:A \ra B$ of subgroups of $H_p$
maps the full 
group $A=H_p$ to the proper subgroup $B=\langle a^ta,t\rangle$.
The set $\nfset_{H/A}:=\{1\} \subseteq \nfset_{H_p}$ is
a transversal for $H/A$.  
Under the map $\ism$ the subgroup $B$  is isomorphic to the 
split extension by $\Z=\langle \hat t\rangle$ of the ideal $I$ 
of $\Z_p[x,1/x]$ generated by $1+x$. This implies that 
$$
H/B\cong\Z_p[x,\frac{1}{x}]/I\cong\Z_p, 
$$
and the set
$$
\nfset_{H/B}:=\{a^\beta\mid\beta\in\Z_p\}
$$
is a set of normal forms of a set of representatives of the cosets 
of $B$ in $H$. 
The corresponding Britton normal form set
for the HNN extension $G_p$ is given by
\begin{eqnarray*}
\nfset_{G_p} &:=& \{s^{-k}a^{\beta_1}sa^{\beta_2}s\ldots sa^{\beta_n}sh
\mid
k,n \ge 0, \beta_i\in\Z_p\text{ for }1\leq i\leq n, \\
&& ~\hspace{.7in} \beta_1 \neq 0 \text{ if both } k>0 \text{ and } n>0,
\text{ and } h\in\nfset_{H_p}\}.
\end{eqnarray*}

Given such a word 
$u=s^{-k}a^{\beta_1}sa^{\beta_2}s\ldots sa^{\beta_n}sh 
\in \nfset_{G_p}$, as in Notation~\ref{britton}
we denote the {\em tail} and {\em head} of $u$ as
$\tail(u) := s^{-k}a^{\beta_1}sa^{\beta_2}s\ldots sa^{\beta_n}s$
and $\head(u):=h$, respectively.
Moreover, let $p_u(x) \in \Z_p[x,\frac{1}{x}]$ and 
$m_u \in \Z$ be defined by
$p_u=0$ and $\rho(\head(u))=\hat t^{m_u}$
in the case that $\head(u)$ is a power of $t$, and
$\rho(\head(u))=p_u(x)\hat t^{m_u}$ otherwise.
Also in the latter case 
let $r_u$ and $l_u$ denote the lowest and
highest degrees, respectively, of monomials
in $p_u(x)$.  Let 
$\alpha_{r_u,u},...,\alpha_{l_u,u}$
(or $\alpha_{r_u},...,\alpha_{l_u}$ when there is
no ambiguity) denote the respective coefficients
in the Laurent polynomial $p_u$.
In the case that $p_u \neq 0$, note that
$\alpha_{l_u,u} \neq 0$;
in the remainder of this proof,
$\alpha_{l_u,u}=0$ implies 
the opposite case that $p_u=0$.

For all $w \in X^*$, let $\nf{w}$ denote
the normal form in $\nfset_{G_p}$ of the element
of $G_p$ represented by $w$.
We note that the language 
$\nfset_{G_p}$ is prefix-closed.
Let $\ga:=\ga(G_p,X)$ be the Cayley graph for
$G_p$ over $X$, let 
$\vec E$ and $\vec P$ be the sets of
directed edges and directed paths in $\ga$, and
let $\tree$ be the maximal tree of $\ga$
corresponding to the set $\nfset_{G_p}$.
For all $u \in \nfset_{G_p}$ and $z \in X$,
let $e_{u,z}$ denote the directed edge in
$\ga$ labeled by $z$ with initial vertex labeled by
the element of $G_p$ represented by $u$.

\bigskip

\noindent{\it Step II. The 
stackable system of \ftp\  van Kampen diagrams for $p<\infty$:}

\smallskip

In this part of the proof we prove that
$G_p$ is stackable over $X$ in the
case when $p$ is finite.  (The case that $p=\infty$ is
similar, and is discussed in Step~IV.)

We obtain the stackable structure
by applying Proposition~\ref{equivalence}
and showing that over the finite presentation
\begin{eqnarray*}
G_p&=&\langle a,s,t\mid a^p=1, [a^t,a]=1, [s,t]=1,
sa^\alpha s^{-1}=ta^\alpha t^{-1}a^\alpha, \\  
&& ~\hspace{1in} sa^\alpha ts^{-1}=a^\alpha ta^\alpha~(\alpha \in \{1,...,p-1\})~\rangle,
\end{eqnarray*}
there is a \ftp\ van Kampen
diagram $\Delta_{u,z}$ with boundary label $uz\nf{uz}^{-1}$
for every normal form word $u \in \nfset_{G_p}$ and generator $z \in X$.
We proceed
via several cases depending upon $u$ and $z$; in each case, we have that 
\ftp\ van Kampen diagrams have been 
constructed for the prior cases.
Also in each case we record the 
corresponding function
$\ff:\vec E \ra \vec P$ on the edge $e_{u,z}$,
and the algorithm to compute this function.

\medskip

\noindent{\bf Case 1}: {\em Suppose that $z=t^{\pm 1}$.}
In this case, either the word $ut^{\pm 1}$ is in normal
form, or else the word $u$ ends with the letter  $t^{\pm 1}$. 
Thus there is a degenerate (and hence \ftp) van Kampen diagram
$\Delta_{u,z}$ with boundary label $uz\nf{uz}^{-1}$,
and $\ff(e_{u,z}):=e_{u,z}$.

\medskip

\noindent{\bf Case 2}: {\em Suppose that
$z=a^{\pm 1}$  and either $p_u(x)=0$ or both $p_u(x) \neq 0$ and
$m_u-l_u\geq 0$.}

\smallskip

\noindent{\bf Case 2.1}: {\em Suppose further that 
$p_u(x) \neq 0$ and $m_u-l_u=0$.}
Then the word $u$ ends with the suffix
$a^{\alpha_{l_u}}$ with $\alpha_{l_u} >0$,
and the last letter of $u$ is $a$.
If $z=a^{-1}$, then 
the normal form of the word $uz$ is
the prefix $u'$ of $u$ satisfying $u=u'a$. 
Thus again
there is a degenerate  diagram
$\Delta_{u,z}$,
and $\ff(e_{u,z}):=e_{u,z}$.
On the other hand, if $z=a$ and
 $\alpha_{l_u}<p-1$, then again the word
$uz$ is in normal form, giving a
degenerate diagram $\Delta_{u,z}$ and $\ff(e_{u,z}):=e_{u,z}$.
Finally, if $z=a$ and $\alpha_{l_u}=p-1$,
we can factor $u=u'a^{p-1}$ for some $u' \in \nfset_{G_p}$,
and so there is a minimal diagram
$\Delta_{u,z}=u'\Delta_{a^{p-1},a}$ 
with a single 2-cell $\Delta_{a^{p-1},a}$ with boundary
label $a^p$; in this case,
$\ff(e_{u,z}):=\path(u,a^{-(p-1)})$.

\smallskip

\noindent{\bf Case 2.2}: {\em Suppose further that either 
$p_u(x) = 0$ or both $p_u(x) \neq 0$ and $m_u-l_u>0$.} 
That is, either $\head(u)=t^{m_u}$, or $\head(u)$
contains the letter $a$ and ends with the letter $t$.  
If $z=a$ then 
the word $uz$ is in normal form
and there is a 
degenerate diagram $\Delta_{u,z}$;  
thus $\ff(e_{u,z}):=e_{u,z}$.
If $z=a^{-1}$, then $\nf{uz}=ua^{p-1}$
and there is a minimal diagram
$\Delta_{u,z}=u\Delta_{1,a^{-1}}$ with a single 2-cell 
$\Delta_{1,a^{-1}}$ with boundary
label $a^p$; in this case,
$\ff(e_{u,z}):=\path(u,a^{p-1})$.

\medskip

\noindent{\bf Case 3}: {\em Suppose that $z=a^{\pm 1}$,
$p_u(x) \neq 0$,  and $m_u-l_u= -1$.} 
%
%
%
Define $\delta \in \{\pm 1\}$ by $z=a^\delta$.
We begin the construction of the \ftp\ van Kampen 
diagram $\Delta_{u,z}$
with an isolated cell with one edge labeled
$z$, and the remaining boundary path
in the other direction labeled
$ta^{-1}t^{-1}a^\delta ta t^{-1}$;
that is, we take
$\ff(e_{u,z}):=\path(u,ta^{-1}t^{-1}a^\delta ta t^{-1})$.
By case 1, we already have \ftp\ diagrams
$\Delta_{u,t}$,
$\Delta_{\nf{uta^{-1}},t^{-1}}$,
$\Delta_{\nf{uta^{-1}t^{-1}a^\delta},t}$,
and
$\Delta_{\nf{uta^{-1}t^{-1}a^\delta ta},t^{-1}}$.
Hence in order to show that we can complete this to a \ftp\  diagram,
it suffices to check that
 \ftp\ diagrams have already been constructed 
for the pairs
$(u_1,a^{-1})$, $(u_2,a^\delta)$, and 
$(u_3,a)$
where
$u_1=\nf{ut}$, 
$u_2=\nf{uta^{-1}t^{-1}}$, and
$u_3=\nf{uta^{-1}t^{-1}a^\delta t}.$ 

Note that the word $u$ ends with the suffix $at^{-1}$.
The word $u_1$ is the prefix of $u$ with the last letter $t^{-1}$
removed, and so $u_1 \in \nfset_{G_p}$ and $u_1$
ends with the letter $a$.
Hence $\Delta_{u_1,a^{-1}}$ was built in case 2.1.

We prove that $\Delta_{u_2,a^\delta}$ has already been
constructed by induction on $\alpha_{l_u,u}$ (and a prior case).
The normal form $u_2$ is obtained from $u$ by removing the 
 last and next-to-last letters $at^{-1}$, and possibly free reduction
(in the case that $\alpha_{l_u,u}=1$ and either
$l_u>r_u$ or $l_u=r_u>0$);
then either ($\head(u_2)=t^{r_u-1}$ and
$\alpha_{l_{u_2},u_2}=0$), or
($\head(u_2)$ contains the
letter $a$ and $m_{u_2}-l_{u_2}=0$), or 
($\head(u_2)$ contains the
letter $a$,
$m_{u_2}-l_{u_2}=-1$ and
$\alpha_{l_{u_2},u_2}<\alpha_{l_u,u}$).
Hence the construction of $\Delta_{u_2,a^\delta}$
follows from case 2 or induction.

Writing 
$u = \tail(u)t^{r_u}a^{\alpha_{r_u,u}} \cdots ta^{\alpha_{l_u,u}}t^{-1}$,
we have
\begin{eqnarray*}
u_3 &=_{G_p}& \tail(u)t^{r_u}a^{\alpha_{r_u,u}} \cdots 
ta^{\alpha_{l_u,u}-1}t^{-1}a^\delta t \\
&=_{G_p}& \tail(u)(t^{r_u}a^{\alpha_{r_u,u}}t^{-r_u}) \cdots 
(t^{l_u}a^{\alpha_{l_u,u}-1}t^{-l_u})(t^{l_u-1}a^\delta t^{-(l_u-1)})t^{l_u} \\
&=_{G_p}& \tail(u)(t^{r_u}a^{\alpha_{r_u,u}}t^{-r_u}) \cdots
(t^{l_u-1}a^{\alpha_{l_u-1,u}+\delta}t^{-(l_u-1)})
  (t^{l_u}a^{\alpha_{l_u,u}-1}t^{-l_u})t^{l_u}.
\end{eqnarray*}
Using the polynomial viewpoint,
$p_{u_3}(x)=p_u(x)+\delta x^{l_u-1} -  x^{l_u}$.
Now either
$\head(u_3)=t^{m_{u_3}}$ (if $p_{u_3}=0$),
or 
$m_{u_3}-l_{u_3} \ge 0$ 
(otherwise);
hence $\Delta_{u_3,a}$ was also constructed in case 2.


\medskip

\noindent{\bf Case 4}: {\em Suppose that $z=s^{-1}$ and $\tail(u)$
does not have a suffix of the form $as$.} 
Either $\tail(u)=s^{-k}$ or $\tail(u)=s$ or
 $\tail(u)$ ends with $s^2$;
 the property that
we exploit in this case is that for all of these options,
we have $\tail(\nf{\tail(u)s^{-1}h})=\nf{\tail(u)s^{-1}}$
and $\head(\nf{\tail(u)s^{-1}h})=\nf{h}$
for all $h \in \{a^{\pm 1},t^{\pm 1}\}^*$.
We proceed by induction on the length $l(\head(u))$ 
of the head of $u$.

Suppose first that $l(\head(u))=0$.  Since $u=\tail(u)$,
either $u=s^{-k}$ (with $k \le 0$), $u=s$, or
$u$ ends with $s^2$.
In all three of these options, either 
$\nf{us^{-1}}=us^{-1}$ or $\nf{us^{-1}}$ 
is the prefix of $u$ obtained
from $u$ by removing a final letter $s$.
Hence there is a degenerate
diagram $\Delta_{u,z}$, and we set
$\ff(e_{u,z}):=e_{u,z}$.

Now suppose that $l(\head(u))>0$, and
write $u=u'z'$ where $z' \in \{a,t^{\pm 1}\}$.

If $z'=t^\delta$ with $\delta \in \{\pm 1\}$,
then we begin the construction of the \ftp\ van Kampen 
diagram $\Delta_{u,z}$
with an isolated cell with 
one edge labeled
$z$, and the remaining boundary path
in the other direction labeled $t^{-\delta}s^{-1}t^\delta$,
and hence set
$\ff(e_{u,z}):=\path(u,t^{-\delta}s^{-1}t^\delta)$.
Now \ftp\ van Kampen diagrams
$\Delta_{u,t^{-\delta}}$ and $\Delta_{\nf{ut^{-\delta}s^{-1}},t^\delta}$
are constructed in case 1,
and since $\nf{ut^{-\delta}}=u'$ satisfies $l(u')=l(u)-1$,
the diagram $\Delta_{\nf{ut^{-\delta}},s^{-1}}$
has been built by induction.

On the other hand if $z'=a$, we build $\Delta_{u,z}$
starting with an isolated cell with 
one edge labeled
$z$, and the remaining boundary path
in the other direction labeled 
$a^{-1}s^{-1}atat^{-1}$,
and so
$\ff(e_{u,z}):=\path(u,a^{-1}s^{-1}atat^{-1})$.
As usual, the required \ftp\ subdiagrams
with isolated edges labeled by $t^{\pm 1}$
have been built in case 1, and the degenerate
diagram $\Delta_{u,a^{-1}}$ is given in case 2.1.
Since $\nf{ua^{-1}}=u'$ is a prefix of $u$,
we again have built the diagram 
$\Delta_{\nf{ua^{-1}},s^{-1}}$
by induction.
So it suffices to show that \ftp\ van Kampen diagrams
have been built for the pairs
$(u_1,a)$, and $(u_2,a)$, where
$u_1=\nf{u a^{-1}s^{-1}}$ and
$u_2=\nf{u a^{-1}s^{-1}at}$.

Since the last letter of $u$ is $z'=a$,
we can write $u=\tail(u)t^{r_u}a^{\alpha_{r_u,u}} \cdots
ta^{\alpha_{l_u,u}}$ where $\alpha_{l_u,u}>0$.
Now
\begin{eqnarray*}
u_1 &=_{G_p}& \tail(u)s^{-1}[s(t^{r_u}a^{\alpha_{r_u,u}} \cdots
ta^{\alpha_{l_u,u}-1} )s^{-1}] \\
&=_{G_p}& \tail(u)s^{-1}[s(t^{r_u}a^{\alpha_{r_u,u}}t^{-r_u}) \cdots
(t^{l_u}a^{\alpha_{l_u,u}-1}t^{-l_u})s^{-1}]t^{l_u}
\end{eqnarray*}
satisfies
$\tail(u_1)=\nf{\tail(u)s^{-1}}$
and 
$$
\head(u_1)=\nf{[s(t^{r_u}a^{\alpha_{r_u,u}}t^{-r_u}) \cdots
(t^{l_u}a^{\alpha_{l_u,u}-1}t^{-l_u})s^{-1}]t^{l_u}}
$$
(applying the property noted at the start of case 4).
Recall that in the polynomial view,
conjugation by $\hat s$ results in multiplication by
$1+x$, and so
$p_{u_1}(x)=[p_u(x)-x^{l_u}](1+x)$
and $\bj(\head(u_1))=p_{u_1}(x)t^{l_u}$.
The polynomial $p_{u_1}(x)$
has degree $l_{u_1}$ that is at most $l_u+1$, and so 
$m_{u_1}-l_{u_1} \ge l_u - (l_u+1)=-1$.
Hence the pair $(u_1,a)$ satisfies the hypotheses
of case 2 or 3, and the
diagram $\Delta_{u_1,a}$ is constructed in one of those cases.

Finally we consider
\begin{eqnarray*}
u_2 &=_{G_p}& \tail(u)s^{-1}[s(t^{r_u}a^{\alpha_{r_u,u}} \cdots
ta^{\alpha_{l_u,u}-1} )s^{-1}]at \\
&=_{G_p}& \tail(u)s^{-1}[s(t^{r_u}a^{\alpha_{r_u,u}}t^{-r_u}) \cdots
(t^{l_u}a^{\alpha_{l_u,u}-1}t^{-l_u})s^{-1}](t^{l_u}at^{-l_u})t^{l_u+1}.
\end{eqnarray*}
Again using the polynomial viewpoint,
we have $p_{u_2}(x)=[p_u(x)-x^{l_u}](1+x)+x^{l_u}$,
which also has degree $l_{u_2} \le l_u+1$,
and $m_{u_2}=l_u+1$,
and so $m_{u_2}-l_{u_2} \ge 0$.
Hence the diagram $\Delta_{u_1,a}$ is constructed in case 2.

\medskip

\noindent{\bf Case 5}: {\em Suppose that $z=s$ and $\head(u)\in\im\atob$.}\label{sIm} 
Similar to the situation in case 4,
a property that
we exploit in this case is that 
$\tail(\nf{\tail(u)sh})=\nf{\tail(u)s}$
and $\head(\nf{\tail(u)sh})=\nf{h}$
for all $h \in \{a^{\pm 1},t^{\pm 1}\}^*$.

Observe that $\head(u)\in\text{Im}\ppi$ implies 
that $p_u(x)=q_u(x)(1+x)$ for some $q_u(x)\in\Z_p[x,\frac{1}{x}]$
that can be written in the form
$$
q_u(x)=\gamma_{r_u,u}x^{r_u}+\cdots+\gamma_{l_u-1,u}x^{l_u-1}
$$
with each $\gamma_{i,u} \in \Z_p$, and
$\gamma_{l_u-1,u}=\alpha_{l_u,u}$.

We proceed by induction on the number $\occ_t(u)$
of occurrences
of $t^{\pm 1}$ in the word $\head(u)$.
Since $x+1$ divides $p_u(x)$,
the polynomial $p_u(x)$ cannot be a single
monomial, and so either $p_u(x)=0$ or 
$|l_u-r_u|>0$.

First suppose that $\occ_t(u)=0$.
Since $p_u(x)$ is not a single monomial,
$\head(u)$ cannot be a nontrivial power of $a$,
and so $\head(u)=\emptyword$
and $u=\tail(u)$.  The normal form $\nf{us}$ is either 
$us$ or the word $u$ with a final letter $s^{-1}$
removed.  In this case 
there is a degenerate
diagram $\Delta_{u,z}$, and we set
$\ff(e_{u,z}):=e_{u,z}$. 

Now suppose that $\occ_t(u)>0$, and
write $u=u'z'$ with $u' \in \nfset_{G_p}$ and
$z' \in \{a,t^{\pm 1}\}$.

If $z'=t^{\delta}$ with $\delta \in \{t^{\pm 1}\}$,
then as in case 4 we begin the 
diagram $\Delta_{u,z}$
with an isolated cell labeled by $zt^{-\delta}s^{-1}t^{\delta}$
and set
$\ff(e_{u,z}):=\path(u,t^{-\delta}st^\delta)$.
The required \ftp\ diagrams
with isolated edges labeled $t^{\pm 1}$
are constructed in case 1, and
since $\nf{ut^{-\delta}}=u'$ satisfies 
$\occ_t(u')=\occ_t(u)-1<\occ_t(u)$, the diagram $\Delta_{u',z}$
has already been built by induction.

Suppose instead that $z'=a$.
Now  
$u=\tail(u)u''t^{\delta}a^{\alpha_{l_u,u}}$
for some $\delta \in \{\pm 1\}$,
$u'' \in \nfset_{H_p}$, and $\alpha_{l_u,u}>0$.  
Recall that we are considering
the case that $p<\infty$ in Step II, and so
$\alpha_{l_u,u} \in \{1,2,...,p-1\}$.
Again using the fact that $p_u(x)$ is
not a single monomial, the exponent $\delta$
must be $1$, and the word $u'' \in \nfset_{H_p}$
contains an occurrence of the letter $a$
and ends either with $a$ or  $t$.

We build $\Delta_{u,z}$ with
an isolated cell labeled 
$zt^{-1}a^{-\alpha_{l_u,u}}s^{-1}a^{\alpha_{l_u,u}}ta^{\alpha_{l_u,u}}$;
then 
$\ff(e_{u,z})=\path(u,a^{-\alpha_{l_u,u}}t^{-1}a^{-\alpha_{l_u,u}}
sa^{\alpha_{l_u,u}}t)$. 
For all $0 \le i \leq \alpha_{l_u,u}-1$,
there is a degenerate diagram
$\Delta_{ua^{-i},a^{-1}}$
from case 2.1,
and by case 1 it remains to show that
we have built \ftp\ diagrams associated to
the pairs 
$(u_{1,i},a^{-1})$, 
$(u_2,s)$, and $(u_{3,i},a)$  for $0 \le i < \alpha_{l_u,u}$, where
$u_{1,i}=\nf{ua^{-\alpha_{l_u,u}}t^{-1}a^{-i}}$, 
$u_2=\nf{ua^{-\alpha_{l_u,u}}t^{-1}a^{-\alpha_{l_u,u}}}$, and
$u_{3,i}=\nf{ua^{-\alpha_{l_u,u}}t^{-1}a^{-\alpha_{l_u,u}}sa^i}$.

Using the factorization of $u$ above,
$u_{1,i}=\nf{\tail(u)u''a^{-i}}$; writing
$u''=u'''a^j$ (with $0 \le j \le p-1$), then
$u_{1,i}=\tail(u)u'''a^{\tilde j}$ where
$\tilde j \in \{0,...,p-1\}$ and $\tilde j \equiv j-i$~(mod~$p$).
This normal form $u_{1,i}$ satisfies either $p_{u_{1,i}}(x)=0$, or else
$p_{u_{1,i}} \neq 0$ and $m_{u_{1,i}}-l_{u_{1,i}} \ge 0$,
and so the diagram $\Delta_{u_{1,i},a^{-1}}$ has been constructed
in case 2.

An analysis of the normal form 
$u_2=\nf{ua^{-\alpha_{l_u,u}}t^{-1}a^{-\alpha_{l_u,u}}}$  yields
\begin{eqnarray*}
u_2 &=_{G_p}& \tail(u)t^{r_u}a^{\alpha_{r_u,u}} \cdots
ta^{\alpha_{l_u,u}}a^{-\alpha_{l_u,u}}t^{-1}a^{-\alpha_{l_u,u}} \\
&=_{G_p}& \tail(u)[(t^{r_u}a^{\alpha_{r_u,u}}t^{-r_u}) \cdots 
(t^{l_u}a^{\alpha_{l_u,u}}t^{-l_u})][(t^{l_u}a^{-\alpha_{l_u,u}}t^{-(l_u)})
(t^{l_u-1}a^{-\alpha_{l_u,u}}t^{-(l_u-1)})]t^{l_u-1}.
\end{eqnarray*}
Then $p_{u_2}=p_u-\alpha_{l_u,u}x^{l_u-1}(x+1)$, and so
$\head(u_2) \in \im\atob$. 
Hence the pair $(u_2,s)$ satisfies the properties
of case 5.  Moreover, since 
$$
u_2=\nf{\tail(u)u''a^{-\alpha_{l_u,u}}}=
\nf{\tail(u)u'''a^ja^{-\alpha_{l_u,u}}}=
\tail(u)u'''a^{\tilde j}
$$
where $\tilde j \in \{0,...,p-1\}$ and 
$\tilde j \equiv j-\alpha_{l_u,u}$~(mod~$p$),
the normal form $u_2$
satisfies 
$\occ_t(u_2)=\occ_t(u)-1<\occ_t(u)$, and the diagram $\Delta_{u_2,s}$
has already been built by induction.

Finally we note that
\begin{eqnarray*}
u_{3,i} &=_{G_p}& u_2sa^i \\
&=_{G_p}& \tail(u)s[s^{-1}(t^{r_u}a^{\alpha_{r_u,u}}t^{-r_u}) \cdots 
(t^{l_u-1}a^{\alpha_{l_u-1,u}-\alpha_{l_u,u}}t^{-(l_u-1)})s]
(t^{l_u-1}a^it^{-(l_u-1)})t^{l_u-1},
\end{eqnarray*}
and so (using the property noted at the beginning
of case 5) $\tail(u_{3,i})=\nf{\tail(u)s}$
and 
$$
\head(u_{3,i})=\nf{[s^{-1}(t^{r_u}a^{\alpha_{r_u,u}}t^{-r_u}) \cdots 
(t^{l_u-1}a^{\alpha_{l_u-1,u}-\alpha_{l_u,u}}t^{-(l_u-1)})s]
(t^{l_u-1}a^it^{-(l_u-1)})t^{l_u-1}}.
$$
In the polynomial viewpoint, the conjugation action of $\hat s^{-1}$
on the element $p_{u_2}=p_u-\alpha_{l_u,u}x^{l_u-1}(x+1)$
of $\im \hat \atob$ is division by $1+x$, and so
$$
[s^{-1}(t^{r_u}a^{\alpha_{r_u,u}}t^{-r_u}) \cdots 
(t^{l_u-1}a^{\alpha_{l_u-1,u}-\alpha_{l_u,u}}t^{-(l_u-1)})s]
=_{G_p} \bj^{-1}(q_u-\alpha_{l_u,u}x^{l_u-1}).
$$
That is, $p_{u_{3,i}}=q_u-\alpha_{l_u,u}x^{l_u-1}+ix^{l_u-1}$. 
Now either $p_{u_{3,i}}=0$ (in the case that
$i=0$ and $p_u=\alpha_{l_u,u}x^{l_u-1}(x+1)$) or
else $p_{u_{3,i}} \neq 0$ and
$l_{u_{3,i}} \le l_u-1 =m_{u_{3,i}}$
(and so $m_{u_{3,i}}-l_{u_{3,i}} \ge 0$). Thus the diagram
$\Delta_{u_{3,i},a}$ is constructed in case 2.

\medskip

\noindent{\bf Case 6}:  {\em Suppose that $z=a^{\pm 1}$,
$p_u(x) \neq 0$, and $m_u-l_u< -1$.}
Write $z=a^\delta$.
The Laurent polynomial 
$p_u(x)=\sum_{i=r_u}^{l_u} \alpha_{i,u}x^i$
can be written 
$$
p_u(x)=q_u(x)(1+x)+R_u
$$ 
for some Laurent polynomial $q_u(x) \in \Z_p[x,\frac{1}{x}]$ and
$R_u \in \Z_p$.  (More specifically, set  $R_u=p_u(-1)$ and let $q_u(x) \in \Z_p[x,\frac{1}{x}]$ be the Laurent polynomial such that
 $p_u(x)-R_p=q_u(x)(x+1)$ which can be found by multiplying $p_u(x)-R_p$ by a suitable power of $x$ and then
 applying the Euclidean
algorithm). Define 
$$
d_\delta(u) = (d_{1,\delta}(u),d_{2,\delta}(u)):=
(|m_u-l_u|,R_u-\delta(-1)^{m_u+1}) \in 
\{i \in \N \mid i \ge 2\} \times \{0,...,p-1\},
$$
where we consider the element $R_u-\delta(-1)^{m_u+1}$
of $\Z_p$ to be an integer in $\{0,...,p-1\} \subset \N_0$.

Consider the lexicographic ordering $\prec$ on 
$\{i \in \N \mid i \ge 2\} \times \{0,...,p-1\} \subset
\N \times \N_0$ (that is,
$(i_1,i_2) \prec (j_1,j_2)$ whenever 
either $i_1<j_1$ or else both $i_1=j_1$ and $i_2<j_2$);
this is a well-founded strict partial ordering.
We proceed by (Noetherian) induction on this ordering on 
$d_\delta(u)$.

For the base case, suppose  that 
$d_{\delta}(u)=(2,0)$.
Then $m_u-l_u=-2$
and $p_u(-1)-\delta(-1)^{m_u+1}=0$.
We begin the diagram $\Delta_{u,z}$ with
 an isolated cell with
an isolated edge labeled $z$, and edge path
in the other direction labeled
$ta^{-\delta}t^{-1}sa^{\delta}s^{-1}$; that is,
$\ff(e_{u,z}):= \path(u,ta^{-\delta}t^{-1}sa^{\delta}s^{-1}).$
The subdiagrams $\Delta_{u,t}$ and $\Delta_{\nf{uta^{-\delta}},t^{-1}}$
are constructed in case 1, and so 
there are four remaining subdiagrams needed to build
$\Delta_{u,z}$ that we need to show have already been 
constructed: 
$\Delta_{u_1,a^{-\delta}}$,
$\Delta_{u_2,s}$, 
$\Delta_{u_3,a^\delta}$,  and 
$\Delta_{u_4,s^{-1}}$, where
$u_1=\nf{ut}$,
$u_2=\nf{uta^{-\delta}t^{-1}}$,
$u_3=\nf{uta^{-\delta}t^{-1}s}$, and
$u_4=\nf{uta^{-\delta}t^{-1}sa^{\delta}}$.

Since $u_1:=\nf{ut}$
satisfies $m_{u_1}-l_{u_1}=-1$, the diagram 
$\Delta_{\nf{ut},a^{-\delta}}$ is built in case 3.

The normal form $u_2:=\nf{uta^{-\delta}t^{-1}}$ satisfies
\begin{eqnarray*}
u_2 &=_{G_p}& \tail(u)(t^{r_u}a^{\alpha_{r_u,u}}t^{-r_u}) \cdots 
 (t^{l_u}a^{\alpha_{l_u,u}}t^{-l_u})t^{m_u} t a^{-\delta} t^{-1}\\
&=_{G_p}& \tail(u)(t^{r_u}a^{\alpha_{r_u,u}}t^{-r_u}) \cdots 
 (t^{l_u}a^{\alpha_{l_u,u}}t^{-l_u})(t^{m_u+1} a^{-\delta}t^{-(m_u+1)})t^{m_u}
\end{eqnarray*}
and so $p_{u_2}(x)=p_u(x)-\delta x^{m_u+1}$.
Now $p_{u_2}(-1)=p_u(-1)-\delta (-1)^{m_u+1}=d_{2,\delta}(u)=0$.
Then $1+x$ divides the Laurent polynomial
$p_{u_2}$; that is, $p_{u_2}$ is in the ideal of $\Z_p[x,\frac{1}{x}]$
generated by $1+x$, and applying the isomorphism
$\ism$ shows that $\head(u_2) \in \im\atob$.
Hence the diagram $\Delta_{u_2,s}$ is built in case 5.

Next $u_3=\nf{uta^{-\delta}t^{-1}s}=\nf{u_2s}$ satisfies
\begin{eqnarray*}
u_3 &=_{G_p}& \tail(u)(t^{r_u}a^{\alpha_{r_u,u}}t^{-r_u}) \cdots 
 (t^{l_u}a^{\alpha_{l_u,u}}t^{-l_u})(t^{m_u+1} a^{-\delta}t^{-(m_u+1)})t^{m_u}s \\
u_3 &=_{G_p}& \tail(u)s[s^{-1}(t^{r_u}a^{\alpha_{r_u,u}}t^{-r_u}) \cdots 
 (t^{l_u}a^{\alpha_{l_u,u}}t^{-l_u})(t^{m_u+1} a^{-\delta}t^{-(m_u+1)})s]t^{m_u}.
\end{eqnarray*}
As in case 5, we note that $\tail(u_3)=\nf{\tail(u)s}$,
and so the normal form of the rest of the expression to
the right of $\tail(u)s$ is the head of $u_3$.
Then $p_{u_3}(x) = (p_u(x)-\delta x^{m_u+1})/(x+1)$;
since $m_u+1 < l_u$, the polynomial $p_{u_3}$
has degree $l_{u_3} = l_u-1$.  Since $m_{u_3}=m_u$,
then $m_{u_3}-l_{u_3} = m_u-l_u+1=-1$, and again
we apply case 3 to show that the diagram $\Delta_{u_3,a^\delta}$
has already been constructed.

Now $u_4 = \nf{u_3a^\delta}$ satisfies
$$
u_4 =_{G_p} \tail(u)s[s^{-1}(t^{r_u}a^{\alpha_{r_u,u}}t^{-r_u}) \cdots 
 (t^{l_u}a^{\alpha_{l_u,u}}t^{-l_u})(t^{m_u+1} a^{-\delta}t^{-(m_u+1)})s][t^{m_u}a^{\delta}t^{-m_u}]t^{m_u}.
$$
Then again $\tail(u_4)=\nf{\tail(u)s}$, and so 
either $\tail(u_4)$ is a power of $s^{\pm 1}$ or ends with $s^2$,
but $\tail(u_4)$ cannot end with $as$.  Hence $\Delta_{u_4,s^{-1}}$
is constructed in case~4.

For the inductive step, suppose that $d_\delta(u) \succ (2,0)$.

Suppose further that $d_{2,\delta}(u)=0$ (and hence
$d_{1,\delta}(u)=|m_u-l_u|>2$).
We follow nearly the same proof as
in the $d_\delta(u)=(2,0)$ (base) case above;
the isolated cell of $\Delta_{u,z}$ is labeled 
$zsa^{-\delta}s^{-1}ta^{\delta}t^{-1}$, and
$\ff(e_{u,z}):= \path(u,ta^{-\delta}t^{-1}sa^{\delta}s^{-1}).$
The only differences with the proof of that base case
is that the applications of case 3 are replaced with
induction.  In particular,
$u_1$ satisfies $m_{u_1}-l_{u_1}=m_u-l_u+1$,
implying that $d_{1,\delta}(u_1)=|m_{u_1}-l_{u_1}|=
|m_u-l_u|+1<|m_u-l_u|=d_{1,\delta}(u)$,
and so $d_\delta(u_1) \prec d_\delta(u)$ and 
the construction of $\Delta_{u_1,a^{-\delta}}$
follows from induction.  Similarly the
fact that $m_{u_3}-l_{u_3} = m_u-l_u+1$ implies that 
$\Delta_{u_3,a^\delta}$ is built by induction.

On the other hand suppose that $d_{2,\delta}(u) > 0$.
(Here we have $d_{1,\delta}(u)=|m_u-l_u| \ge 2$.)
Let 
$$
\eta:=(-1)^{m_u+1}.
$$
The construction of $\Delta_{u,z}$ begins
with an isolated cell having
isolated edge labeled $z$, and edge path
in the other direction labeled
$ta^{-\eta}t^{-1}a^\delta ta^{\eta}t^{-1}$
then
$\ff(e_{u,z}):=\path(u,ta^{-\eta}t^{-1}a^\delta ta^{\eta}t^{-1})$.
Applying case 1 to obtain the required (degenerate)
diagrams with isolated edges labeled $t^{\pm 1}$,
we have left to check whether the three \ftp\ diagrams 
$\Delta_{u_1,a^{-\eta}}$,
$\Delta_{\tilde u_2,a^\delta}$ and, 
$\Delta_{\tilde u_3,a^{\eta}}$
have already been constructed, where
$u_1=\nf{ut}$,
$\tilde u_2=\nf{uta^{-\eta}t^{-1}}$, and
$\tilde u_3=\nf{uta^{-\eta}t^{-1}a^\delta t}$.

For the first of these, the arguments
above show that $|m_{u_1}-l_{u_1}|=|m_u-l_u|-1$,
and so $d_\delta(u_1) \prec d_\delta(u)$; 
the construction of $\Delta_{u_1,a^{-\eta}}$
follows from case 3 or induction.

Replacement of  $\delta$ by $\eta$ in
the computation for $u_2$ above shows that
$$\tilde u_2 =_{G_p} \tail(u)(t^{r_u}a^{\alpha_{r_u,u}}t^{-r_u}) \cdots 
 (t^{l_u}a^{\alpha_{l_u,u}}t^{-l_u})(t^{m_u+1} a^{-\eta}t^{-(m_u+1)})t^{m_u}
$$
and $p_{\tilde u_2}(x)=p_u(x)-\eta x^{m_u+1}$.
Again using the fact that $m_u+1<l_u$,
then $p_{\tilde u_2}(x) \neq 0$,
$m_{\tilde u_2}=m_u$, and $l_{\tilde u_2} = l_u$.  
We have
$d_{1,\delta}(\tilde u_2) = d_{1,\delta}(u)$, and
$R_{\tilde u_2}=p_{\tilde u_2}(-1)=
p_u(-1)-\eta (-1)^{m_u+1}=R_u-1$. 
Then
$$
d_{2,\delta}(\tilde u_2) = R_{\tilde u_2}-\delta(-1)^{m_{\tilde u_2}}=
R_u-1-\delta (-1)^{m_u}=d_{2,\delta}(u)-1< d_{2,\delta}(u).
$$
Therefore 
$d_\delta(\tilde u_2) \prec d_\delta(u)$ and the construction
of the diagram $\Delta_{\tilde u_2,a^\delta}$
follows from induction.

Finally we consider
\begin{eqnarray*}
\tilde u_3 &=_{G_p}& \tail(u)(t^{r_u}a^{\alpha_{r_u,u}}t^{-r_u}) \cdots 
 (t^{l_u}a^{\alpha_{l_u,u}}t^{-l_u})t^{m_u}ta^{-\eta}t^{-1}a^{\delta} t\\
&=_{G_p}& \tail(u)(t^{r_u}a^{\alpha_{r_u,u}}t^{-r_u}) \cdots 
 (t^{l_u}a^{\alpha_{l_u,u}}t^{-l_u})
    (t^{m_u+1}a^{-\eta}t^{-(m_u+1)})(t^{m_u}a^\delta t^{-m_u})t^{m_u+1},
\end{eqnarray*}
which has associated polynomial
$p_{\tilde u_3}(x)=p_u(x)-\eta x^{m_u+1}+\delta x^{m_u}$.
Here 
$m_{\tilde u_3}=m_u+1$ and $l_{\tilde u_3}=l_u$,
and so
$d_{1,\delta}(\tilde u_3) < d_{1,\delta}(u)$ and
$d_\delta(\tilde u_2) \prec d_\delta(u)$, and so the construction
of $\Delta_{\tilde u_3,a^{\eta}}$
follows from induction.

\medskip

\noindent{\bf Case 7}: {\em Suppose that $z=s^{-1}$ and $\tail(u)$ has a
suffix $as$.} 
We proceed by induction on the length $l(\head(u))$.

Suppose first that $l(\head(u))=0$.
Then $\nf{us^{-1}}=u'$,
and so there is a degenerate diagram $\Delta_{u,z}$ and
$\ff(e_{u,z})=e_{u,z}$.

Suppose next that $l(\head(u))>0$.
Write $u=u'z'$ with $u' \in \nfset_{G_p}$ and
$z' \in \{a,t^{\pm 1}\}$.

If $z'=t^{\delta}$ with $\delta \in \{t^{\pm 1}\}$,
the construction and inductive proof are identical to
the $z'=t^{\delta}$ subcase of case 4,
with
$\ff(e_{u,z}):=\path(u,t^{-\delta}st^\delta)$.

On the other hand suppose that 
$z'=a$.
We begin building $\Delta_{u,z}$ with an
isolated cell labeled 
$zta^{-1}t^{-1}a^{-1}sa$;
that is,
$\ff(e{u,z})=\path(u,a^{-1}s^{-1}atat^{-1})$.
Now case 1 provides \ftp\ van Kampen diagrams
$\Delta_{\hat u,t^{\delta}}$
and cases 2,3, and 6 provide \ftp\ van Kampen
diagrams $\Delta_{\hat u,a^{\delta}}$
for all $\hat u \in \nfset_{G_p}$ and $\delta \in \{\pm 1\}$,
yielding diagrams
$\Delta_{u,a^{-1}}$,
$\Delta_{\nf{ua^{-1}s^{-1}},a}$,
$\Delta_{\nf{ua^{-1}s^{-1}a},t}$,
$\Delta_{\nf{ua^{-1}s^{-1}at},a}$, and
$\Delta_{\nf{ua^{-1}s^{-1}ata},t^{-1}}$.
Since $\nf{ua^{-1}}=u'$ satisfies 
$l(\head(u'))=l(\head(u))-1<l(\head(u))$,
the diagram $\Delta_{\nf{ua^{-1}},s^{-1}}$
has also already been constructed, by induction.

\medskip

\noindent{\bf Case 8}: {\em Suppose that $z=s$ and $\head(u)\not\in\im\atob$.}
As in case 5, we proceed by induction on the number $\occ_t(u)$
of occurrences
of $t^{\pm 1}$ in $\head(u)$.

Suppose first that $\occ_t(u)=0$.  
Then $\head(u)=a^{\alpha_{l_u}}$, and since
$\head(u) \notin\im\atob$, then $\alpha_{l_u}>0$.  
In this case the word $us$
is in normal form, and so we have a degenerate
diagram $\Delta_{u,z}$ and $\ff(e_{u,z})=e_{u,z}$.

Now suppose that $\occ_t(u)>0$, and
write $u=u'z'$ with $u' \in \nfset_{G_p}$ and
$z' \in \{a,t^{\pm 1}\}$.

If $z'=t^{\delta}$ with $\delta \in \{t^{\pm 1}\}$,
the construction of $\Delta_{u,z}$ has isolated
cell labeled $zt^{-\delta}s^{-1}t^\delta$,
with
$\ff(e_{u,z}):=\path(u,t^{-\delta}st^\delta)$.
Since $\nf{ut^{-\delta}}=u'$ satisfies
$\occ_t(u')=\occ_t(u)-1<\occ_t(u)$, the diagram $\Delta_{u',z}$
has already been constructed by induction.

Suppose instead that $z'=a$.
Now  
$\head(u)=u''t^{\delta}a^{\alpha_{l_u}}$
for some $\delta \in \{\pm 1\}$ and
$u'' \in \nfset_{H_p}$.  Again we recall that
$p<\infty$ in Step II, and so
$\alpha_{l_u} \in \{1,2,...,p-1\}$.

If $\delta=1$,
we build $\Delta_{u,z}$ with
an isolated cell labeled 
$zt^{-1}a^{-\alpha_{l_u}}s^{-1}a^{\alpha_{l_u}}ta^{\alpha_{l_u}}$;
then 
$\ff(e_{u,z})=\path(u,a^{-\alpha_{l_u}}t^{-1}a^{-\alpha_{l_u}}
sa^{\alpha_{l_u}}t)$.
Cases 1,2,3, and 6 provide \ftp\ van Kampen
diagrams corresponding to the edges labeled
$t^{\pm 1}$ and $a^{\pm 1}$ along the boundary of
this isolated cell.  Now the normal form
$\tilde u:=\nf{ua^{-\alpha_{l_u}}t^{-1}a^{-\alpha_{l_u}}}=
\nf{u''a^{-\alpha_{l_u}}}$ is obtained
from $u''a^{-\alpha_{l_u}}$ by possible reduction
modulo $p$ of a final power of $a$; then $\tilde u$
satisfies $\occ_t(\tilde u)=\occ_t(u)-1<\occ_t(u)$,
and again induction applies to show that the
diagram $\Delta_{\tilde u,s}$ has already been constructed.

If $\delta=-1$, we 
build $\Delta_{u,z}$ with
an isolated cell labeled 
$za^{-\alpha_{l_u}}t^{-1}s^{-1}a^{\alpha_{l_u}}t^{-1}a^{\alpha_{l_u}}$;
then 
$\ff(e_{u,z})=\path(u,a^{-\alpha_{l_u}}ta^{-\alpha_{l_u}}
st^{-1}a^{\alpha_{l_u}})$.  As before,
cases 1,2,3, and 6 provide \ftp\ van Kampen
diagrams corresponding to the edges labeled
$t^{\pm 1}$ and $a^{\pm 1}$.  The normal form
$\hat u:=\nf{ua^{-\alpha_{l_u}}ta^{-\alpha_{l_u}}}=
u''a^{-\alpha_{l_u}}$,
and so $\occ_t(\tilde u)<\occ_t(u)$.
Once again induction applies to show that the
diagram $\Delta_{\hat u,s}$ has already been constructed,
as required.

\bigskip

\noindent{\it Step III. Autostackability of $G_p$ for $p<\infty$:}
\smallskip

In this section again we consider the $p<\infty$ case.
Throughout this step we will repeatedly apply the
closure properties of regular and \syreg\ languages
discussed in Section~\ref{sec:notation}.

Before analyzing the graph of the 
stacking map, we first discuss the
set $\nfset_{G_p}$ of normal forms.
Let
\begin{eqnarray*}
\Tail &:=& \{s^{-k}a^{\beta_1}sa^{\beta_2}s\ldots sa^{\beta_n}s
\mid
k,n \ge 0, \beta_i\in\Z_p\text{ for }1\leq i\leq n, \\
&& ~\hspace{.7in} \text{ and }\beta_1 \neq 0 
\text{ if both } k>0 \text{ and } n>0\};
\end{eqnarray*}
that is, $\Tail$ is the set of tails of elements of 
$\nfset_{G_p}$.  Then
$$
\Tail = (s^{-1})^*(\{\emptyword,a,a^2,...,a^{p-1}\}s)^*
\setminus X^*s^{-1}sX^*
$$
is built from finite subsets of $X^*$
using concatenation, complement, and Kleene star operations, 
and hence is a regular language.  Similarly the 
normal form set $\nfset_{H_p}$ 
for $H_p$ can be written
$$
\nfset_{H_p} = 
(t^* \cup (t^{-1})^*)(\{\emptyword,a,a^2,...,a^{p-1}\}t)^*
\{\emptyword,a,a^2,...,a^{p-1}\}(t^* \cup (t^{-1})^*)
\setminus X^*\{tt^{-1},t^{-1}t\}X^*
$$
and so $\nfset_{H_p}$ is also regular.
Finally the normal form set $\nfset_{G_p}$
is the concatenation
$\nfset_{G_p}=\Tail \cdot \nfset_{H_p}$ of these 
two regular languages, and therefore
$\nfset_{G_p}$ is also regular.

The stacking map associated to the 
flow function $\ff$ in Step II of this proof is given by
$$
\sff(u,z):=
\begin{cases}
%
%
z   & \text{if } z=t^{\pm 1}
\end{cases}
$$
from case 1,
$$
\sff(u,z):=
\begin{cases}
%
%
z   & \text{if } z=a \text{ and } 
                    u \in \Tail \cdot (t^{-1})^* \cup X^*t 
                   \cup (X^*a \setminus X^*a^{p-1}) \\
a^{-(p-1)}
    & \text{if } z=a \text{ and } 
                    u \in X^*a^{p-1} \\
z   & \text{if } z=a^{-1} \text{ and } u \in X^*a \\
a^{p-1}
    & \text{if } z=a^{-1} \text{ and } 
                    u \in \Tail \cdot (t^{-1})^* \cup X^*t \\
ta^{-1}t^{-1}a^\delta ta t^{-1}
    & \text{if } z=a^{\delta},~\delta \in \{\pm 1\}, \text{ and } 
                    u \in X^*at^{-1} \\
%
%
ta^{-\delta}t^{-1}sa^{\delta}s^{-1}
    & \text{if } z=a^{\delta},~\delta \in \{\pm 1\}, \text{ and } 
                    u \in X^*aY^*t^{-2} \\
& ~\hspace{.4in} \text{ and }
                    p_u(-1)-\delta(-1)^{m_u+1} \equiv 0 (\text{mod }p)\\
ta^{-\eta}t^{-1}a^\delta ta^{\eta}t^{-1}
    & \text{if } z=a^{\delta},~\delta \in \{\pm 1\}, \text{ and } 
                    u \in X^*aY^*t^{-2}, \\
& ~\hspace{.4in} 
                    p_u(-1)-\delta(-1)^{m_u+1} \not\equiv 0 (\text{mod }p), \\
& ~\hspace{.4in}             \text{ and } \eta=(-1)^{m_u+1}  

\end{cases}
$$
from cases 2,3, and 6 (where we recall that $Y=\{a^{\pm 1},t^{\pm 1}\}$),
$$
\sff(u,z):=
\begin{cases}
%
%
z   & \text{if } z=s^{-1} \text{ and }
                    u \in \Tail \\
t^{-\delta}s^{-1}t^\delta
    & \text{if } z=s^{-1} \text{ and } 
                    u \in X^*t^\delta,~\delta \in \{\pm 1\} \\
a^{-1}s^{-1}atat^{-1}
    & \text{if } z=s^{-1} \text{ and } 
                    u \in X^*a \\
\end{cases}
$$
from cases 4 and 7, and
$$
\sff(u,z):=
\begin{cases}
%
%
z   & \text{if } z=s \text{ and }
                    u \in \Tail  \cdot a^*\\ 
t^{-\delta}st^\delta
    & \text{if } z=s \text{ and } 
                    u \in X^*t^\delta,~\delta \in \{\pm 1\} \\
a^{-\alpha}t^{-1}a^{-\alpha}sa^{\alpha}t
    & \text{if } z=s \text{ and } 
                    u \in X^*ta^{\alpha},~\alpha 
                     \in \{1,...,p-1\} \\
a^{-\alpha}ta^{-\alpha} st^{-1}a^{\alpha}
    & \text{if } z=s \text{ and }  
                    u \in X^*t^{-1}a^{\alpha},~\alpha 
                     \in \{1,...,p-1\} 
\end{cases}
$$
from cases 5 and 8.
The graph $\graph(\sff)$ of this function is a union of
15 languages, one for each of the pieces of
this piecewise defined function.  
We consider each of these languages in order.

The first language is
$L_1=\cup_{\delta \in \{\pm 1\}} \nfset_{G_p} \times \{t^{\delta}\}
\times \{t^{\delta}\}$ (from case 1).  
Since this is a product of regular
languages, $L_1$ is regular.
The next five languages,
arising from cases 2 and 3, are
\begin{eqnarray*}
L_2 &=& [\nfset_{G_p} \cap (\Tail \cdot (t^{-1})^* \cup X^*t 
                   \cup (X^*a \setminus X^*a^{p-1}))] 
\times \{a\}  \times \{a\}, \\
L_3&=& [\nfset_{G_p} \cap X^*a^{p-1}] \times \{a\} \times \{a^{-(p-1)}\}, \\
L_4&=& [\nfset_{G_p} \cap X^*a] \times \{a^{-1}\} \times \{a^{-1}\}, \\
L_5&=& [\nfset_{G_p} \cap \Tail \cdot (t^{-1})^* \cup X^*t]
 \times \{a^{-1}\} \times \{a^{p-1}\}, \hspace{.2in} \text{and}\\
L_6&=& \cup_{\delta \in \{\pm 1\}}
       [\nfset_{G_p} \cap X^*at^{-1} ] \times \{a^{\delta}\}
 \times \{ta^{-1}t^{-1}a^\delta ta t^{-1}\},
\end{eqnarray*}
and regularity of $\nfset_{G_p}$ together with 
the closure properties show that all five are \syreg.
Similarly the last 7 languages, corresponding 
to the graph of the pieces of $\sff$ defined
in cases 4,7,5, and 8, are
\begin{eqnarray*}
L_9 &=& \Tail \times \{s^{-1}\} \times \{s^{-1}\},\\
L_{10} &=& \cup_{\delta \in \{\pm 1\}}
      [\nfset_{G_p} \cap X^*t^\delta]
        \times \{s^{-1}\} \times \{t^{-\delta}s^{-1}t^\delta\}, \\
L_{11} &=& [\nfset_{G_p} \cap X^*a] \times \{s^{-1}\} \times 
       \{a^{-1}s^{-1}atat^{-1}\}, \\
L_{12} &=& [\Tail  \cdot a^*] \times \{s\} \times \{s\}, \\
L_{13} &=& \cup_{\delta \in \{\pm 1\}}
      [\nfset_{G_p} \cap X^*t^\delta]
        \times \{s\} \times \{t^{-\delta}st^\delta\}, \\
L_{14} &=& \cup_{\alpha \in \{1,...,p-1\}}
      [\nfset_{G_p} \cap X^*ta^{\alpha}] \times \{s\}
      \times \{a^{-\alpha}t^{-1}a^{-\alpha}sa^{\alpha}t\},
    \hspace{.2in} \text{and} \\
L_{15} &=& \cup_{\alpha \in \{1,...,p-1\}}
      [\nfset_{G_p} \cap X^*t^{-1}a^{\alpha}] \times \{s\}
      \times \{a^{-\alpha}ta^{-\alpha} st^{-1}a^{\alpha}\}.
\end{eqnarray*}
Regularity of the languages $\Tail$ and $\nfset_{G_p}$
and closure properties also show that these
7 languages are \syreg.

The remaining two subsets of $\graph(\sff)$ arise 
from case 6, namely
\begin{eqnarray*}
L_{7} &=& 
   \cup_{\delta \in \{\pm 1\}}\cup_{\eta \in \{\pm 1\}}
      [X^*aY^*t^{-2} \cap 
     \Tail \cdot (M_{\eta} \cap N_{\delta,\eta})] \\
&& \hspace{2in}
    \times \{a^{\delta}\} \times \{ta^{-\delta}t^{-1}sa^{\delta}s^{-1}\}
  \hspace{.1in} \text{ and} \\
L_8 &=& 
   \cup_{\delta \in \{\pm 1\}}\cup_{\eta \in \{\pm 1\}}
      [X^*aY^*t^{-2} \cap 
     \Tail \cdot (M_{\eta} \cap (\nfset_{H_p} \setminus N_{\delta,\eta}))] \\
&& \hspace{2in}
  \times \{a^{\delta}\} \times \{ta^{-\eta}t^{-1}a^\delta ta^{\eta}t^{-1}\},
\end{eqnarray*}
where
\begin{eqnarray*}
M_{\eta}&=&\{u \in \nfset_{H_p}  \mid
     (-1)^{m_u+1}=\eta\},
\text{ and} \\ 
N_{\delta,\eta}&=&
\{u \in \nfset_{H_p} \cap Y^*aY^*t^{-1}\mid 
  p_u(-1) \equiv \delta\eta(\text{mod }p)\}.
\end{eqnarray*}
In order to show that $L_7$ and $L_8$ are \syreg,
it suffices to show that the languages 
$M_\eta$ and $N_{\delta,\eta}$ are regular.

For $u \in \nfset_{H_p}$, the integer
$m_u$ is the sum of the exponents of 
the letters $t^{\pm 1}$ in $u$. That is,
$M_1$ is the set of words in $\nfset_{H_p}$
with $t$-exponent sum odd, and
 $M_{-1}$ is the set of words in $\nfset_{H_p}$
with $t$-exponent sum even.
Let $\gamma:Y^* \ra \Z/2$ be the monoid homomorphism
to the finite monoid $Z/2$
defined by $\gamma(a)=\gamma(a^{-1})=0$
and $\gamma(t)=\gamma(t^{-1})=1$.  The
preimage sets $\gamma^{-1}(\{0\})$
and $\gamma^{-1}(\{1\})$ are regular sets
(see Section~\ref{sec:notation}).  Then 
$M_1=\nfset_{H_p} \cap \gamma^{-1}(\{1\})$
and 
$M_{-1}=\nfset_{H_p} \cap \gamma^{-1}(\{0\})$
are intersections of regular languages,
and hence $M_\eta$ is regular for $\eta \in \{\pm 1\}$.

Given any $u \in \nfset_{H_p} \cap Y^*aY^*t^{-1}$,
as usual we write 
$p_u(x)=\alpha_{r_u,u}x^{r_u}+ \cdots +\alpha_{l_u,u}x^{l_u}$,
where
$u=t^{r_u}a^{\alpha_{r_u,u}} \cdots t a^{\alpha_{l_u,u}} t^{-l_u+m_u}$.
Then
$
p_u(-1)  
=(-1)^{r_u}(\alpha_{r_u,u}+ \cdots +\alpha_{l_u,u}(-1)^{l_u-r_u}.
$
Splitting this into separate cases depending
upon whether $r_u$ is even or odd yields
$$
N_{\delta,\eta}=
[((t^2)^* \cup (t^{-2})^*)\widetilde N_{\delta,\eta}(t^{-1})^*]
~\cup~
[(t(t^2)^* \cup t^{-1}(t^{-2})^*)\widetilde N_{-\delta,\eta}(t^{-1})^*]
$$
where 
\begin{eqnarray*}
\widetilde N_{\delta,\eta} &:=& 
\{a^{\alpha_0}ta^{\alpha_{1}} \cdots t a^{\alpha_{l}}t^{-1}
\mid 0 \le l,~\alpha_{i} \in \{0,1,...,p-1\} \text{ for }
0 \le i \le l, \\
&& \hspace{.5in} \alpha_0,\alpha_l \neq 0, 
\text{ and }
\alpha_0-\alpha_1+\cdots +(-1)^{\alpha_l} \equiv \delta\eta(\text{mod }p)
\}.  
\end{eqnarray*}
It now suffices to show that $\widetilde N_{\delta,\eta}$
is regular.

In order to show that $\widetilde N_{\delta,\eta}$
is regular, we describe a finite
state automaton (FSA) accepting this language.  (We refer
the reader to~\cite{echlpt} and~\cite{hu} for
details on the definition of a finite state automaton.)
The alphabet of the FSA is the set $\{a,t,t^{-1}\}$, and the
set of states (i.e., the finite memory)
of the automaton is
$$
Q = \big( \Z_p \times \{y,n\} \times \{\pm 1\} \times \Z_p \big) 
\cup \{I,F\}.
$$
The initial state is $I$,
and the set of accept states is
$\{\delta\eta\} \times \{y\} \times  \{\pm 1\} \times \{0\}$. 
The transition function $d:Q \times \{a,t^{\pm 1}\} \ra Q$,
which determines the state that the FSA
moves to depending upon its current state and
the next letter that it reads, is defined
by $d(I,a):=(0,n,1,1)$, $d(I,t^{\pm 1})=F$,
$d((i,n,\mu,j),a):=(i,n,\mu,j+1)$ if $0 \le j<p-1$,
$d((i,n,\mu,p-1),a):=F$,
$d((i,n,\mu,j),t):=(i+\mu j(\text{mod }p),n,-\mu,0)$,
$d((i,n,\mu,0),t^{-1}):=F$,
$d((i,n,\mu,j),t^{-1}):=(i+\mu j(\text{mod }p),y,-\mu,0)$ 
if $0< j \le p-1$, and
$d((i,y,\mu,j),z):=F$
and $d(F,z):=F$ for all $z \in \{a,t^{\pm 1}\}$.
That is, $F$ is a ``failure'' state, and 
for the states other than $I,F$,
the first copy of $\Z_p$ records the
alternating sum $\alpha_0-\alpha_1+...$ computed
so far, the $y/n$ records whether a $t^{-1}$ letter
has occurred, the $\pm 1$ records whether
the current $a^{\alpha_i}$ word being read has
$(-1)^i$ equal to 1 or -1, and the final copy of $\Z_p$
records the number of $a$ letters that have been
read since the last occurrence of a letter $t$.
A word $u$ is accepted by this FSA if and
only if it starts with the letter $a$,
contains no substring $a^p$ or $tt^{-1}$, ends with the
letter $t^{-1}$ but contains no other
occurrence of that letter, and
satisfies $p_u(-1) \equiv \delta\eta(\text{mod }p)$.
This in turn holds iff
$u$ lies in $\widetilde N_{\delta,\eta}$.

We now have that all of the languages $L_1,...,L_{15}$ 
are \syreg.  Therefore
their union
$\graph(\sff)=\cup_{i=1}^{15} L_i$ is also \syreg,
as required.

\bigskip

\noindent{\em Step IV.  Algorithmically stackable system
of \ftp\ van Kampen diagrams for $p=\infty$:}
\smallskip

In this part of the proof we consider the case that $p=\infty$.
The construction
of the flow function can be done in a very similar way to the proofs in earlier steps so we leave the details
to the reader and just state the stacking map that one gets over the finite presentation
$$
G_\infty=\langle a,s,t\mid a^p=1, [a^t,a]=1, [s,t]=1,
sas^{-1}=ta t^{-1}a, sa ts^{-1}= ata\rangle.
$$

$$
\sff(u,z):=
\begin{cases}
%
%
z   & \text{if } z=t^{\pm 1}\\
%
%
z   & \text{if } z=a^{\pm 1} \text{ and } 
                    u \in \Tail \cdot (t^{-1})^* \cup X^*t 
                   \cup X^*a^{\pm 1}  \\
ta^{-\epsilon}t^{-1}a^{\delta}ta^{\epsilon}t^{-1}
   & \text{if } z=a^{\delta}, \delta,\epsilon \in \{\pm 1\}, \text{ and }
     u \in X^* a^{\epsilon}t^{-1} \\
%
%
ta^{-\delta}t^{-1}sa^{\delta}s^{-1}
    & \text{if } z=a^{\delta},~\delta \in \{\pm 1\}, \text{ and } 
                    u \in X^*a^{\pm 1}Y^*t^{-2} \\
& ~\hspace{.4in} \text{ and }
                    p_u(-1)-\delta(-1)^{m_u+1} = 0 \\
ta^{-\eta}t^{-1}a^\delta ta^{\eta}t^{-1}
    & \text{if } z=a^{\delta},~\delta \in \{\pm 1\}, \text{ and } 
                    u \in X^*a^{\pm 1}Y^*t^{-2}, \\
& ~\hspace{.4in} 
                    p_u(-1)-\delta(-1)^{m_u+1} \neq 0 \\
& ~\hspace{.4in}             \text{ and } \eta=(-1)^{m_u+1} \text{ if } p_u(-1)-\delta(-1)^{m_u+1} >0, \\
 & ~\hspace{.4in}          \eta=(-1)^{m_u} \text{ otherwise},\\
%
%
z   & \text{if } z=s^{-1} \text{ and }
                    u \in \Tail \\
t^{-\delta}s^{-1}t^\delta
    & \text{if } z=s^{-1} \text{ and } 
                    u \in X^*t^\delta,~\delta \in \{\pm 1\} \\
a^{-\delta}s^{-1}a^{\delta}ta^{\delta}t^{-1}
    & \text{if } z=s^{-1} \text{ and } 
                    u \in X^*a^{\delta}, \delta \in \{\pm 1\} \\
%
%
z   & \text{if } z=s \text{ and }
                    u \in \Tail  \cdot (a^* \cup (a^{-1})^*)\\ 
t^{-\delta}st^\delta
    & \text{if } z=s \text{ and } 
                    u \in X^*t^\delta,~\delta \in \{\pm 1\} \\
a^{-\delta}t^{-1}a^{-\delta}sa^{\delta}t
    & \text{if } z=s \text{ and } 
                    u \in X^*ta^{\delta}(a^{\delta})^*,~\delta 
                     \in \{\pm 1\} \\
a^{-\delta}ta^{-\delta} st^{-1}a^{\delta}
    & \text{if } z=s \text{ and }  
                    u \in X^*t^{-1}a^{\delta}(a^{\delta})^*,~\delta 
                     \in \{\pm 1\} 
\end{cases}
$$

\end{proof}

As noted in Section~\ref{sec:intro},
the following is immediate from Theorem~\ref{thm:metabelian}.

\begin{corollary}\label{cor:metabelian}
The class of autostackable groups contains
nonconstructible metabelian groups.
\end{corollary}

\section*{Acknowledgment}

The first author was partially supported by grants from the 
National Science Foundation (DMS-1313559) and the
Simons Foundation (\#245625) and the second author was supported by Gobierno de Aragon, European Regional Development Funds and MTM2015-67781-P.



\begin{thebibliography}{99}

\bibitem{bnw}
Bartholdi, L., Neuhauser, M., and Woess, W.,
{\em Horocyclic products of trees},
J.~Eur.~Math.~Soc.~(JEMS) {\bf 10} (2008), 771--816.

\bibitem{baumslagnon} 
Baumslag, G.,
{\em A non-cyclic one-relator group all of whose finite quotients are cyclic}, 
J.~Austral.~Math.~Soc.~{\bf 10} (1969), 497--498.

\bibitem{Baumslag1} Baumslag, G., 
{\em A finitely presented metabelian group with a
free abelian derived group of infinite rank}, 
Proc.~Amer.~Math.~Soc.~{\bf 35} (1972), 61--62.


\bibitem{bv}
Bogopolski, O.,~and Ventura, E., 
{\em A recursive presentation for Mihailova's subgroup}, 
Groups Geom.~Dyn.~{\bf 4} (2010),  407--417. 

\bibitem{bhtame} M.~Brittenham and S.~Hermiller, 
{\em Tame filling invariants for groups},
Internat.~J.~Algebra Comput.~{\bf 25} (2015), 813--854.

\bibitem{MarkSusan} M.~Brittenham and S.~Hermiller, 
{\em A uniform model for almost convexity and rewriting systems},
J.~Group Theory {\bf 18} (2015), 805--828.

\bibitem{bhh} M.~Brittenham, S.~Hermiller and D.~Holt, 
{\em Algorithms and topology for Cayley graphs of groups},
J.~Algebra {\bf 415} (2014), 112--136. 

\bibitem{bhj} M.~Brittenham, S.~Hermiller and A.~Johnson,
{\em Homology and closure properties of autostackable groups}, 
J. Algebra {\bf 452} (2016),  596--617.


\bibitem{bht} M.~Brittenham, S.~Hermiller and T.~Susse,
{\em Geometry of the word problem for 3-manifold groups}, 
preprint, 2016.



\bibitem{echlpt}
Epstein, D.B.A., Cannon, J., Holt, D., 
Levy, S., Paterson, M.~and Thurston, W.,
{Word Processing in Groups},
Jones and Bartlett, Boston, 1992.


\bibitem{der}
Dison, W., Einstein,, E., and Riley, T.R.,
{\em Taming the hydra: the word problem and extreme integer compression},
arXiv:1509.02557. 

\bibitem{disonriley}
Dison, W.~and Riley, T.R.,
{\em Hydra groups},
Comment.~Math.~Helv.~{\bf 88} (2013), 
507--540. 

\bibitem{gersten}
Gersten, S.M.,
{\em Dehn functions and $\ell_1$-norms of finite presentations},
in: Algorithms and Classification in Combinatorial Group Theory, 
195--225, Springer, Berlin, 1992.


\bibitem{grovessmith}
Groves, J.R.J. and Smith, G.C.,
{\em Soluble groups with a finite rewriting system},
Proc.~Edinburgh Math.~Soc.~{\bf 36} (1993), 283--288. 

\bibitem{hu}
Hopcroft, J.~and Ullman, J.D.,
Introduction to automata theory, languages, and computation,
Addison-Wesley Series in Computer Science, 
Addison-Wesley Publishing Co., Reading, Mass., 1979.


\bibitem{lyndonschupp} Lyndon, R.C.~and Schupp, P.E.,
{ Combinatorial group theory},
Classics in Mathematics, Springer-Verlag, Berlin, 2001.

\bibitem{mihailova}
Mihailova, K.A.,
{\em The occurrence problem for direct products of groups}, 
Dokl.~Acad.~Nauk SSRR {\bf 119} (1958), 1103--1105.


\bibitem{platonov}  Platonov, A.N.,
{\em An isoparametric function of the Baumslag-Gersten group},
Vestnik Moskov.~Univ.~Ser.~I Mat.~Mekh.~{\bf 2004}, 12--17, 70; 
English. trans.~Moscow Univ.~Math.~Bull.~{\bf 59} (2004), 
12--17.

\bibitem{stw}
Stein, M., Taback, J., and Wong, P.,
{\em Automorphisms of higher rank lamplighter groups},
Internat.~J.~Algebra Comput.~{\bf 25} (2015), 1275--1299. 


\end{thebibliography}
\end{document}